\documentclass[12pt,oneside]{amsart}

\usepackage{amsmath} 
\usepackage{amssymb}
\usepackage{amsfonts}

\numberwithin{equation}{section}

\newtheorem{thm}{Theorem}[section]
\newtheorem{lem}[thm]{Lemma}

\newtheorem{prop}[thm]{Proposition}
\newtheorem{rem}[thm]{Remark}

\theoremstyle{definition}
\newtheorem{defn}[thm]{Definition}
\newtheorem{exe}[thm]{Example}
\newtheorem{claim}[thm]{Claim}  
\theoremstyle{remark}

\newcommand{\br}{\mathbf R}
\newcommand{\bc}{\mathbf C}
\newcommand{\bz}{\mathbf Z}
\newcommand{\Cal}{\mathcal}
\newcommand{\ddt}{\partial_t }
\newcommand{\dtt}{{\textstyle\frac{d}{dt}}}
\newcommand{\id}{\operatorname{Id}}
\newcommand{\im}{\operatorname{Im}}

\renewcommand{\dim}{\operatorname{Dim}}
\renewcommand{\ker}{\operatorname{Ker}}
\newcommand{\coker}{\operatorname{Coker}}
\newcommand{\op}{\operatorname{Op}}
\newcommand{\ol}{\overline}

\newcommand{\pb}[1]{\left\{\,#1\,\right\}}
\newcommand{\ran}{\operatorname{Ran}}
\newcommand{\re}{\operatorname{Re}}
\newcommand{\restr}[1]{\big|_{#1}}
\newcommand{\set}[1]{\left\{\,#1\,\right\}}

\newcommand{\sgn}{\operatorname{sgn}}
\newcommand{\sh}{\sharp}
\newcommand{\supp}{\operatorname{\rm supp}}
\newcommand{\mn}[1]{\Vert#1\Vert}
\newcommand{\w}[1]{\langle #1\rangle }
\newcommand{\sw}[1]{\left ( #1\right ) }
\newcommand{\wf}{\operatorname{WF}}
\newcommand{\wt}{\widetilde}

\begin{document}

\baselineskip 20 pt 
\lineskip 2pt
\lineskiplimit 2pt
\parskip 2pt

\title[Solvability]{On the solvability of systems of\\ pseudodifferential operators}
\author[NILS DENCKER]{{\textsc Nils Dencker}}
\address{Centre for Mathematical Sciences, University of Lund, Box 118,
SE-221 00 Lund, Sweden}
\email{dencker@maths.lth.se} 
\date{February 25, 2010}
\dedicatory{Dedicated to Hans Duistermaat on his sixtyfifth birthday}



\maketitle

\section{Introduction}\label{intro}

In this paper we shall study the question of local solvability for square
systems of classical pseudodifferential operators $P \in {\Psi}^m_{cl}(M)$ on a
$C^\infty$ manifold $M$. We shall only consider operators acting on
distributions $\Cal D'(M, \bc^N)$ with values in $\bc^N$ but since the results are local
and invariant under base changes, they
immediately carry over to operators on sections of vector bundles.
We shall assume that the symbol of $P$ is an
asymptotic sum of homogeneous $N \times N$ matrices, with homogeneous
principal symbol $p = {\sigma}(P)$. We shall also assume that $P$ is of 
principal type, so that the principal symbol vanishes of first order
on the kernel, see Definition~~\ref{princtype}.

Local solvability for a $N \times N$ system $P$ at  a compact set $K \subseteq M$ means that
the equation 
\begin{equation}\label{locsolv}
Pu = v 
\end{equation}
has a local weak solution $u \in \Cal D'(M, \bc^N)$ in a neighborhood of $K$
for all $v\in C^\infty(M, \bc^N)$ in a subset of finite codimension.  
We say that $P$ is microlocally solvable at a compactly based cone
$K \subset T^*M$ if there exists an integer $N$ such that for
every $f \in H_{(N)}^{loc}(M, \bc^N)$ there exists $u \in \Cal D'(M, \bc^N)$ so that $K \bigcap
\wf (Pu-f) = \emptyset$, see ~\cite[Definition~26.4.3]{ho:yellow}. Here
$H_{(s)}$ is the usual $L^2$ Sobolev space 
and $ H_{(s)}^{loc}$ is the localized Sobolev
space, i.e., those $f \in \Cal D'$ such that ${\phi}f \in H_{(s)}$ for any
${\phi} \in C^\infty_0$.

Hans Lewy's famous counterexample~\cite{lewy} from 1957 showed that not all
smooth linear differential operators are solvable.
It was conjectured by Nirenberg and Treves~ \cite{nt} in 1970 that
local solvability for principal type {\em scalar}
pseudodifferential operators is equivalent to condition~ (${\Psi}$) on
the principal symbol $p$,
which means that
\begin{multline}\label{psicond} \text{$\im (ap)$
    does not change sign from $-$ to $+$}\\ 
 \text{along the oriented
    bicharacteristics of $\re (ap)$}
\end{multline}
for any $0 \ne a \in C^\infty(T^*M)$. Recall that the
operator is of principal type if $dp \ne 0$ when $p=0$, and the
oriented bicharacteristics are 
the positive flow-outs of the Hamilton vector field 
$$H_{\re (ap)} = \sum_j \partial_{{\xi}_j}\re (ap) \partial_{x_j} -
\partial_{x_j} \re(ap) \partial_{{\xi}_j} $$ 
on $\re (ap) =0$ (also called the semibicharacteristics of ~$p$). 
Condition~\eqref{psicond} is obviously invariant under symplectic changes of
coordinates and multiplication with non-vanishing factors.
Thus the condition is invariant under conjugation of ~$P$ with elliptic
Fourier integral operators. It actually suffices to check the
condition with some $0 \ne a \in C^\infty$ such that $ d(\re ap) \ne
0$, see \cite[Lemma~26.4.10]{ho:yellow}. Recall that $p$ satisfies condition~ ($\ol {\Psi}$) if $\ol p$
satisfies condition~ (${\Psi}$), and that $p$ satisfies
condition ($P$) if there are no sign changes on the semibicharacteristics, that is, $p$ satisfies
both condition ~(${\Psi}$) and~ ($\ol {\Psi}$).

The necessity of (${\Psi}$) for local solvability of scalar
pseudodifferential operators was proved by Moyer \cite{mo:solv} in
1978 for the two dimensional case, and by H\"ormander \cite{ho:suff} in
1981 for the general case. The sufficiency of condition
(${\Psi}$) for solvability of scalar pseudodifferential operators in two
dimensions was proved by Lerner \cite{ln:2d} in ~1988.
The Nirenberg-Treves conjecture was finally proved by the author~ \cite{de:NT}, giving
solvability with a loss of two derivatives (compared with the elliptic
case). This has been improved to a loss of arbitrarily more than $3/2$
derivatives by the author ~ \cite{de:pisa}, and to a loss of exactly $3/2$ by
Lerner~\cite{ln:cutloss}. Observe that there only exist counterexamples showing a loss
of $1 + {\varepsilon}$ derivatives for arbitrarily small
${\varepsilon} > 0$, see Lerner~\cite{ln:ex}.

For partial differential operators, condition ~(${\Psi}$) is equivalent to
condition~ ($P$). The sufficiency of ($P$) for local solvability of
scalar partial differential operators with a loss of one derivative
was proved in 1973 by Beals and Fefferman ~\cite{bf},
introducing the Beals-Fefferman calculus.
In the case of operators which are {\em not} of principal type,
conditions corresponding to~(${\Psi}$) are neither necessary nor sufficient for local
solvability, see~\cite{CPT}. 

For systems there is no corresponding conjecture for solvability. By
looking at diagonal systems, one finds that condition ~(${\Psi}$) for
the eigenvalues of the 
principal symbol is necessary for solvability. But when the principal
symbol is not diagonalizable, condition ~(${\Psi}$) is not sufficient, see
Example~\ref{popex} below. It is not even known if condition ~(${\Psi}$)
is sufficient in the case when the principal symbol is  $C^\infty$
diagonalizable. We shall consider
the case of when the principal symbol has constant characteristics,
then the eigenvalue close to the origin has constant 
multiplicity, see Definition~\ref{constchar}. In that case, the eigenvalue is a
$C^\infty$ function and condition~ (${\Psi}$) on the eigenvalues is well-defined.
The main result of the paper is that classical square systems of pseudodifferential operators
of principal type having constant characteristics are solvable (with a
loss of $3/2$ derivatives) if and only if the eigenvalues of the
principal symbol satisfies condition~(${\Psi}$), see Theorem~\ref{mainthm}.

\section{Statement of results}\label{statement}

We say that the system 
$P \in {\Psi}^m_{cl}$ is {\em classical} if the symbol of $P$ is an
asymptotic sum $P_m + P_{m-1} + \dots $ where $P_j(x,{\xi})$ is homogeneous 
of degree $j$ in ~~${\xi}$, here $P_m$ is called the {\em principal
symbol} of $P$. Recall that the
eigenvalues of the principal symbol are the solutions to the characteristic equation
\begin{equation*}
 |P_m(x,{\xi}) - {\lambda}\id_N| = 0
\end{equation*}
where $|A|$ is the determinant of the matrix.
In the following, we shall denote by
$\ker A$ the kernel and $\ran A$ the range of the matrix~~$A$.
The definition of principal type for systems is similar
to the one for scalar operators. 

\begin{defn} \label{princtype}
We say that the $N \times N$ system $P(w) \in C^1$ is of {\em principal type} at $w_0$ if 
\begin{equation}\label{pr_type}
\partial_{\nu}P(w_0):\ \ker P(w_0) \mapsto \coker
P(w_0)  = \bc^N/\ran P(w_0) 
\end{equation} 
is bijective for some ${\nu}$, here $\partial_{\nu}P = \w{{\nu}, dP}$
and the mapping~\eqref{pr_type} is given by $u \mapsto
\partial_{\nu}P(w_0) u$ modulo $\ran P(w_0)$. 
We say that  $P \in {\Psi}^m_{cl}$ is of principal type at $w_0$ if the
principal symbol  $P_m(x,{\xi})$ is of principal type at $w_0$.
\end{defn}

\begin{rem} 
If $P(w) \in C^1$ is of principal type and $A(w)$, $B(w) \in
C^1$ are invertible then $APB$ is of principal
type. We also have that $P$ is of principal type if and only if the adjoint $P^*$
is of principal type.
\end{rem}

In fact, Leibniz' rule gives
\begin{equation}\label{invcomp}
 d(APB) = (dA)PB + A(dP)B + APdB
\end{equation}
and $\ran (APB) = A(\ran P)$ and $\ker (APB) =
B^{-1}(\ker P)$ when $A$ and $B$ are invertible, which gives the invariance 
under left and right multiplication. Since  $ \ker P^*(w_0) = \ran
P(w_0)^\bot$ we find that $P$ satisfies ~\eqref{pr_type} if and 
only if  
\begin{equation}\label{bilform}
 \ker P(w_0) \times \ker P^*(w_0) \ni (u,v) \mapsto
\w{\partial_{\nu}P(w_0) u,v}
\end{equation}
is a non-degenerate bilinear form. Since $\w{\partial_{\nu}P^* u, v} =
\ol{\w{\partial_{\nu}P v, u}}$ we then obtain that~ $P^*$ 
is of principal type.

Observe that only square systems can be of
principal type since 
\begin{equation*}
 \dim \ker P = \dim \coker P + M - N
\end{equation*}
if $P$ is an $N \times M$ system.
In general, if the system is of principal type and has constant multiplicity of the
eigenvalues then there are no
non-trivial Jordan boxes, see Definition~\ref{multdef} and
Proposition~\ref{princprop}. Then we also have that the eigenvalues 
${\lambda}$ are of principal type: $d{\lambda} \ne 0$ when ${\lambda}
= 0$. When the multiplicity is
equal to one, this condition is sufficient. In fact, by using the
spectral projection one can find invertible systems~ $A$ and~
$B$ so that
\begin{equation*}
 APB = 
\begin{pmatrix}
{\lambda} & 0\\ 0 & E 
\end{pmatrix}
\end{equation*}
with $E$ invertible $(N-1) \times (N-1)$ system, and this system is
obviously of principal type.

\begin{defn}\label{multdef}
Let $A\/$ be an $N \times N$ matrix and ${\lambda}$ an
eigenvalue of $A$.
The multiplicity of ~~ ${\lambda}$ as a root of the characteristic equation $|A -
{\lambda}\id_N| = 0$ is called the {\em algebraic}  multiplicity of the eigenvalue,
and the dimension of $\ker (A - {\lambda}\id_N)$ is called the {\em
geometric} multiplicity. 
\end{defn}

Observe that if the matrix $P(w)$ depend continuously on a parameter
$w$, then the eigenvalues ${\lambda}(w)$ also depend continuously on ~$w$. 
Such a continuous function~${\lambda}(w)$ of eigenvalues we
will call a {\em section of eigenvalues of}~$P(w)$. 

\begin{rem} \label{smoothev}
If the section of eigenvalues ${\lambda}(w)$ of the $N \times N$ system
$P(w) \in C^\infty$ has 
constant {\em algebraic} multiplicity then ${\lambda}(w) \in C^\infty$.
In fact, if $k$ is the multiplicity then ${\lambda} =
{\lambda}(w)$ solves $\partial_{\lambda}^{k-1}|P(w)
- {\lambda}\id_N| = 0$ so ${\lambda}(w) \in C^\infty$ by
the Implicit Function Theorem.  
\end{rem}

This is {\em not} true for constant geometric multiplicity, for
example $P(t) = 
\begin{pmatrix} 
0 & 1 \\ t & 0 
\end{pmatrix}$, $t \in \br$, has geometric multiplicity equal to one for the
eigenvalues $\pm\sqrt{t}$. 
Observe the geometric multiplicity is lower or equal to the
algebraic, and for symmetric systems they are equal.
We shall assume that  the eigenvalues close to zero have constant algebraic
and geometric multiplicities by the following definition.

\begin{defn} \label{constchar}
The $N \times N$ system $P(w) \in C^\infty$ has {\em
constant characteristics} near $w_0$ if there exists an ${\varepsilon} > 0$ so that
any section of eigenvalues~${\lambda}(w)$ of~$P(w)$
with $|{\lambda}(w)| < {\varepsilon}$ has both constant algebraic and
geometric multiplicity when $|w - w_0| < {\varepsilon}$. 
\end{defn}

Definition~\ref{constchar} is invariant under changes of bases: $P
\mapsto E^{-1}PE$ where $E$ is an invertible system, since this
preserves the multiplicities of the eigenvalues of the system. 
It is also invariant under taking adjoints, since 
$|P^*(w) - {\lambda}^*(w)\id| = \ol{|P(w) - {\lambda}(w)\id|}$ and 
$\dim \ker (P^*(w) - {\lambda}^*(w)\id) = \dim \ker (P(w) - {\lambda}(w)\id)$.
The definition is {\em not} invariant under multiplication of the
system with invertible systems, 
even in the case when $P(w) = {\lambda}(w)\id$ since $A(w)P(w) = {\lambda}(w)A(w)$
need not have constant characteristics.

Observe that generically the eigenvalues of a system have constant multiplicity, but
not necessarily when equal to zero. For example, the system  
\begin{equation*}
 P({w}) = 
\begin{pmatrix}
{w}_1 & {w}_2 \\
{w}_2 & - {w}_1 
\end{pmatrix}
\end{equation*}
is symmetric and of principal type with eigenvalues $\pm\sqrt{{w}_1^2 +
{w}_2^2}$, which have constant multiplicity except when equal to 0.

\begin{defn}
Let the $N \times N$ system $P \in {\Psi}^m_{cl}$ be of principal type and
constant characteristics. We say that $P$ satisfies
condition (${\Psi}$) or ~($P$) if the eigenvalues of the principal
symbol satisfies condition (${\Psi}$)  or ~($P$).
\end{defn}

Observe that the eigenvalue close to the
origin is a uniquely defined $C^\infty$ function of principal type by
Definition~\ref{multdef} and Proposition~\ref{princprop}. 
Thus, the semibicharacteristics of the
eigenvalues are well-defined near the characteristic set $\set{w: \
  |P(w)|= 0}$, so the conditions (${\Psi}$) and~($P$) on the
eigenvalues are well-defined. Also well-defined is the condition that
the Hamilton vector field of an eigenvalue~$\lambda$ does not have the
radial direction when $\lambda = 0$.

To get local solvability at a point $x_0 \in M$ we
shall also assume a strong form of the non-trapping condition at
~$x_0$ for the eigenvalues ${\lambda}$ of ~~$P$:
\begin{equation}\label{nontrap}
 {\lambda} = 0 \implies \partial_{\xi} {\lambda}\ne 0
\end{equation}
This means that all non-trivial
semibicharacteristics of ${\lambda}$ are transversal to the fiber $T^*_{x_0}M$,
which originally was the condition for principal type of Nirenberg and
Treves ~\cite{nt}. Microlocally, in a conical neighborhood of
a $(x,{\xi}) \in T^*M$, we can always obtain \eqref{nontrap}
after a canonical transformation. In the following, we shall use  the
usual $L^2$ Sobolev norm $\mn{u}_{(s)}$ and the $L^2$ norm $\mn u = \mn u_{(0)}$.

\begin{thm}\label{mainthm}
Let $P \in {\Psi}^m_{cl}(M)$ be an $N \times N$ system 
of principal type and constant characteristics 
near $(x_0,{\xi}_0) \in T^*M$, such that the Hamilton vector field of
an eigenvalue~$\lambda$ does not have the radial direction when
$\lambda = 0$.   
Then $P$ is microlocally solvable near $(x_0,{\xi}_0)$ if and
only if condition~$({\Psi})$ is satisfied  near ~$(x_0,{\xi}_0)$,
and then
\begin{equation}\label{solvest1}
\mn{u} \le  C(\mn{P^*u}_{(3/2-m)} + \mn{R u} + \mn u_{(-1)})\qquad u
\in  C_0^\infty(M, \bc^N)
\end{equation}
Here $R \in {\Psi}^{1/2}_{1, 0}(M)$ is a $K \times N$ system such that
$(x_0,{\xi}_0) \notin \wf (R)$, which gives microlocal
solvability of ~$P$ at ~$(x_0,{\xi}_0)$ with a loss of at most $3/2$ derivatives.
If the eigenvalues also satisfy ~\eqref{nontrap} at $x_0 \in M$, then we obtain 
~\eqref{solvest1} with  $x \ne x_0$
in $\wf (R)$, which gives  local solvability of
~$P$ at ~$x_0$ with a loss of at most $3/2$ derivatives.
\end{thm}

As usual, $\wf (R)$ is the smallest smallest conical set in
$T^*M\setminus 0$ such
that $R \in {\Psi}^{-\infty}$ in the complement.   
The conditions in Theorem~\ref{mainthm} are invariant under
conjugation with scalar Fourier 
integral operators since they only depend on the principal symbol of
the system.  They are also invariant under the base change: $P
\mapsto E^{-1}PE$ with invertible system ~ $E$, since this
preserves the eigenvalues of the principal symbol.
The conditions of Theorem~\ref{mainthm} are more or less necessary, of
course condition 
(${\Psi}$) is necessary even in the scalar case. Example ~\ref{popex}
shows that the condition of principal type is necessary in the
case of constant characteristics, and  Example~\ref{newex} shows that constant characteristics is
necessary for solvability for systems of principal type.

We shall postpone the proof of Theorem~\ref{mainthm} to
Section~\ref{proof}. The proof of the necessity is essentially 
the classical Moyer-H\"ormander proof for the scalar case.
The proof of the sufficiency will be an 
adaption of the proof for the scalar case in ~ \cite{de:NT}, using
some of the ideas of Lerner~\cite{ln:cutloss}. In fact, since the normal
form of the operator will have a scalar principal symbol,
the multiplier will essentially be the same as in  ~ \cite{de:NT}.
But since we lose more than one derivative in the estimate we also
have to consider the lower order matrix valued terms in the expansion of
the operator. This is done in
Section~\ref{lower} and is the main new part of the paper. 
In Section~\ref{reduc} we review the Weyl calculus and state the
estimates we will use in the proof of
Theorem~\ref{mainthm}. But we shall postpone the proof of the
semiclassical estimate of Proposition~\ref{mainprop} until Section~\ref{lower}.
In  Section~\ref{proof} we prove  Theorem~\ref{mainthm} by a microlocal
reduction to a normal form using the estimates in
Section~\ref{reduc}. In Section~\ref{symb} we define the symbol
classes and weights we are going to use.  In Section~\ref{norm} we review the Wick
quantization, introduce the function spaces and the multiplier 
estimate that we will use for the proof of
Proposition~\ref{mainprop}.
Finally, in Section~\ref{lower} we prove Proposition~\ref{mainprop}
by estimating the contributions of the lower order terms.
The proof of Theorem~\ref{mainthm} in
Section~\ref{proof} also gives the following results. 

\begin{rem}\label{prem}
If $P$ is of principal type with constant characteristics
satisfying condition ~($P$) then we get the estimate~\eqref{solvest1}
with $3/2$ replaced by $1$. If $P$ satisfies  condition~$(\ol {\Psi})$ and
some repeated Poisson bracket of the real and imaginary parts 
of the eigenvalue  close to the origin is non-vanishing, then we obtain a subelliptic
estimate for ~$P$ with $3/2$ replaced by $k/k+1$ in~\eqref{solvest1} for some $k\in \bz_+$,
see ~\cite[Chapter~27]{ho:yellow}. 
\end{rem}

The Poisson bracket of $f$ and $g$ is defined by $\set{f,g} = H_fg$.
Theorem~\ref{mainthm} has applications to scalar non-principal type pseudodifferential
operators by the following result.

\begin{thm}\label{exthm}
Let   $Q \in {\Psi}^1_{cl}(M)$ be a scalar operator of principal type
near~$(x_0,{\xi}_0)\in T^*M$ and let $A_j \in {\Psi}^0_{cl}(M)$, $j=1,
\dots, N$ be scalar. Then the equation
\begin{equation}\label{scalareq}
 Pu = Q^N u + \sum_{j=0}^{N-1}A_j Q^ju = f
\end{equation}
is locally solvable near~$(x_0,{\xi}_0)$ if and only if ${\sigma}(Q)$ satisfies
condition ~$({\Psi})$  near~$(x_0,{\xi}_0)$.
\end{thm}

\begin{proof} 
This is a standard reduction to a first order system. For scalar $u \in
\Cal D'$ we let $u_{j+1} = Q^ju$ for $0 \le j < N$. Then
\eqref{scalareq} holds if and only if
$U = {}^{t}(u_1, \dots, u_N)$ solves
\begin{equation}\label{syseq}
 \mathbb P\,U  = F
\end{equation}
where 
\begin{equation*}
  \mathbb P = 
\begin{pmatrix}
Q & -1 & 0 & 0 & \dots \\ 
0 & Q & -1 & 0 & \dots \\
0 & 0 & Q & -1 & \dots\\
\dots &&&&\\
A_0 & A_1 & A_2 & \dots & Q+A_{N-1} 
\end{pmatrix}
\end{equation*}
and $F = {}^{t}(0, 0, \dots, f)$. Now the
equation~\eqref{scalareq} is locally solvable if and only if the
system~\eqref{syseq} is locally solvable. In fact, to solve~\eqref{syseq} we
first put $u_1 = 0$, $u_2 = -f_1$ and recursively $u_{j+1} =
Qu_j-f_j$ for $1 \le j < N$. Then we only have to solve
~\eqref{scalareq} for $u = v_1$ with $f$ 
depending on $f_j$, and add $ Q^{j-1}v_1$ to $u_j$.
Now ${\sigma}(\mathbb P) = {\sigma}(Q)\id_N$ which is of principal
type with constant characteristics so it is locally solvable if and
only if $Q$ satisfies condition~$({\Psi})$ according to Theorem~\ref{mainthm}.
\end{proof}

We shall conclude the section with some examples. But first we prove a
result about the characterization of systems of principal type.

\begin{prop}\label{princprop} 
Assume that $P(w) \in C^\infty$ is an $N \times N$ system such that
$|P(w_0)|  = 0$ and there exists an ${\varepsilon} > 0$ such that the
eigenvalue~${\lambda}$ of~$P(w)$  
with $|{\lambda}| < {\varepsilon}$ has constant algebraic multiplicity
in a neighborhood of ~$w_0$. Let ${\lambda}(w) \in C^\infty$ be the unique
eigenvalue for $P(w)$ near $w_0$ satisfying 
${\lambda}(w_0) = 0$ by Remark~\ref{smoothev}. Then $P(w)$ is of
principal type at ~$w_0$ if and only if 
$d{\lambda}(w_0) \ne 0$ and the geometric multiplicity of the eigenvalue~
${\lambda}$ is equal to the algebraic multiplicity at $w_0$.
\end{prop}

Thus, if $P(w)$ is of principal type having constant
characteristics, then all sections of eigenvalues
${\lambda}(w)$ are of principal type and we have no non-trivial
Jordan boxes in the normal form. This means that for symmetric systems having constant
characteristics it suffices that the eigenvalues are of principal
type. If $P(w)$ does not have constant
characteristics then this is no longer true, in fact the eigenvalues need not
even be differentiable, see Example~\ref{nonsolvex}. 

Observe that if $P(w)$ is of principal type and has constant
characteristics, then $P(w) - {\lambda}\id_N$ is of principal type
near~$w_0$ for $|{\lambda}| \ll 1$. In fact, the algebraic and
geometric multiplicities are constant for the eigenvalue ${\lambda}$
and $d{\lambda} \ne 0$ near ~$w_0$.

Now the eigenvalue ${\lambda}(w)$ in
Proposition~\ref{princprop} is the unique $C^\infty$ solution
to $\partial^{k-1}_{\lambda} |P(w) -
{\lambda}\id_N| = 0$ according
to Remark~\ref{smoothev}, where $k$ is the algebraic multiplicity. Thus we find that 
$d{\lambda}(w) \ne 0$ if and only if 
$$\partial_w\partial^{k-1}_{\lambda}
|P(w) - {\lambda}\id_N| \ne 0 \qquad \text{when ${\lambda} = {\lambda}(w)$}$$ 
We only need this condition for a symmetric systems having constant
multiplicity to be of principal type.

\begin{exe}
Let 
\begin{equation*}
 P({w}) = 
\begin{pmatrix}
 {w}_1 + i {w}_2^2 & {w}_2 \\ 0 & {w}_1 + i {w}_2^2
\end{pmatrix} \qquad w = (w_1,w_2) \in \br^2
\end{equation*}
then $P$ is of principal type, has constant algebraic multiplicity of the
eigenvalue ${w}_1 + i {w}_2^2$ but not constant 
geometric multiplicity. In fact, $\partial_{{w}_1}P = \id_2$,
$P({w})$ has non-trivial kernel
only when ${w}_2 = 0$ but the geometric multiplicity of the eigenvalue
is equal to one when ${w}_2 \ne 0$.
\end{exe}

\begin{exe}
Let 
\begin{equation*}
 P = p(x,D_x)\id_N + B(x,D_x) + P_0(x,D_x)
\end{equation*}
where $p\in S^1_{cl}$ is a scalar homogeneous symbol of principal type, $B \in
{\Psi}^1_{cl}$  with nilpotent homogeneous principal symbol ${\sigma}(B)$ and $P_0  \in 
{\Psi}^0_{cl}$. Then $p$ is the only eigenvalue to
${\sigma}(P)$ and $P$ is of principal type if and only if 
${\sigma}(B) = 0$ when $p = 0$ by Proposition~\ref{princprop}.
\end{exe}

\begin{rem}
Observe that the conclusion of  Proposition~\ref{princprop} does not
hold if the algebraic multiplicity is not constant. For example
\begin{equation*}
 P({w}) = 
\begin{pmatrix}
{w}_1 & 1 \\ {w}_2 & {w}_1 
\end{pmatrix} \qquad w = (w_1,w_2) \in \br^2
\end{equation*}
has determinant is equal to $ {w}_1^2 - {w}_2 $ and eigenvalues
${w}_1 \pm \sqrt{{w}_2}$, so the geometric but not the algebraic 
multiplicity is constant near ${w}_2 = 0$. 
Since 
\begin{equation*}
 \begin{pmatrix}
 0 & 1 \\ 1 & 0
\end{pmatrix}
P({w}) = 
\begin{pmatrix}
{w}_2 & {w}_1 \\ {w}_1 & 1 
\end{pmatrix}
\end{equation*}
we find that $P({w})$ is of principal type at $(0,0)$ by the invariance.
\end{rem}

\begin{proof}[Proof of Proposition~\ref{princprop}]
First we note that $P(w)$ is of principal typ at $w_0$ if and only if
\begin{equation}\label{princtypecond}
 \partial_{\nu}^k|P(w_0)| \ne 0 \qquad k = \dim \ker P(w_0)
\end{equation}
for some ${\nu} \in T(T^*\br^n)$. Observe that $ \partial^j|P(w_0)| =
0$ for $j < \dim \ker P(w_0)$. In fact, by choosing bases for $\ker
P(w_0)$ and $\im P(w_0)$ respectively, and extending to bases of $\bc^N$, we obtain
matrices $A$ and $B$ so that 
\begin{equation*}
 AP(w)B = 
\begin{pmatrix}
P_{11}(w) & P_{12}(w) \\ P_{21}(w) & P_{22}(w)
\end{pmatrix}
\end{equation*}
where $|P_{22}(w_0)| \ne 0$ and $P_{11}$, $P_{12}$ and $P_{21}$ all vanish at
~$w_0$. By the invariance, $P$ is of principal type if and
only if $\partial_{\nu}P_{11}$ is invertible for some ~${\nu}$, so by
expanding the determinant we obtain~\eqref{princtypecond}.

Now since the eigenvalue ${\lambda}(w)$ has constant algebraic multiplicity near $w_0$, we find that
\begin{equation*}
 |P(w) - {\lambda}\id_N| = ({\lambda}(w)-{\lambda})^m e(w,{\lambda})
\end{equation*}
near $w_0$, where ${\lambda}(w_0) = 0$, $e(w,{\lambda}) \ne 0$ and $m \ge \dim \ker P(w_0)$ is the
algebraic multiplicity. By putting ${\lambda}= 0$ we obtain that
$ \partial_{\nu}^j|P(w_0)| = 0$ if $j < m$ and $
\partial_{\nu}^m|P(w_0)| = (\partial_{\nu}{\lambda}(w_0))^m e(w_0,0)$
which proves Proposition~\ref{princprop}.
\end{proof}

The following example shows that if the system is not of
principal type then it need not be solvable, even if it has
real eigenvalues with constant characteristics.

\begin{exe}\label{popex} 
Let $P \in {\Psi}^2_{cl}$ have principal symbol ${\sigma}(P) = p^2$ where
$p \in S^1_{cl}$ is real, homogeneous of degree 1, and of principal
type. Then $P$ is not solvable 
if the imaginary part of the subprincipal symbol of $P$ changes sign
on the bicharacteristics of $p$ by~ \cite{pop}. Observe that the
subprincipal symbol is invariantly defined at the double
characteristics $p^{-1}(0)$.
As in the proof of Theorem~\ref{exthm}, the equation can be reduced to the system
\begin{equation*}
\mathbb P =
 \begin{pmatrix}
p(x,D_x)  & p_1(x,D_x) \\
-1 &  p(x,D_x)
\end{pmatrix} 
\end{equation*}
where ${\sigma}(p_1) \in S^1_{cl}$ is equal to the subprincipal symbol of $P$ on
$p^{-1}(0)$. The system is of principal type at $w\in p^{-1}(0)$ if
and only if ${\sigma}(p_1)(w) = 0$ by Proposition~\ref{princprop}.
This system has real eigenvalues of constant characteristics so it
satisfies condition ($P$), but it is neither solvable nor of principal type if $
\im p_1$ changes sign along the bicharacteristic of ~$p$. 
Observe that the system 
\begin{equation*}
\mathbb P =
 \begin{pmatrix}
p(x,D_x)  & p_1(x,D_x) \\
0 &  p(x,D_x)
\end{pmatrix} 
\end{equation*}
is solvable, since it is upper triangle with solvable
diagonal elements. Thus the solvability depends on the lower order
terms in this case.
\end{exe}

The next example is an unsolvable system of principal type with
real eigenvalues, but it does not have constant characteristics. 

\begin{exe}\label{newex}
Let 
\begin{equation}\label{Pdef}
P = 
\begin{pmatrix}
D_{x_1} & B(x,D_x) \\ -1 & D_{x_1} + R(D_x)
\end{pmatrix}
\end{equation}
 where $R({\xi}) = {\xi}_2^{2}/|{\xi}|$ and ${\sigma}(B)(x,{\xi}) =
{\xi}_2 B_0(x,{\xi})$ with homogeneous $B_0 \in S^0$.
The eigenvalues of the principal symbol ${\sigma}(P)$ are ${\xi}_1$
and ${\xi}_1 + R({\xi})$ which are real and coincide when ${\xi}_2 = 0$.
Since $\partial_{{\xi}_1}\sigma (P) = \id_2$ and ${\sigma}(B)$ vanish
when ${\xi}_2 = 0$, 
we find that $P$ is of principal type by Proposition~\ref{princprop}. If
$t \mapsto \im B_0(t,x',0,0,{\xi}'')$ changes sign at $t= x_1$, then $P$ is not
microlocally solvable at $(x,0,0,{\xi}'')$, here $x= (x_1,x') =
(x_1,x_2,x'')$ and ${\xi}'' \ne 0$.
In fact, the system $PU = F$ with $U = {}^{t}(u_1,u_2)$ and $F =
{}^{t}(f_1,f_2)$ is equivalent to the equation
\begin{equation}\label{Qdef}
Qu_2 = (D_{x_1}(D_{x_1} + R(D_x)) + B(x,D_x))u_2 = f_1 + D_{x_1}f_2
\end{equation}
if we put $u_1 = (D_{x_1} + R(D_x))u_2 - f_2$. Thus the system~ $P$ is
solvable if and only if $Q$ is solvable.
That $Q$ is not solvable follows from using the construction of
approximate solutions to the adjoint in
~\cite{MenUhl}, replacing $D_{x_2}$ with $R(D_x)$.

We can also generalize this to the case where 
\begin{equation*}
 R({\xi}) = {\xi}_2^k|{\xi}|^{1-k}
\end{equation*}
and ${\sigma}(B)(x,{\xi}) = {\xi}^j_2 B_j(x,{\xi})$ with $j < k$ and $B_j \in
S^{1-j}$ homogeneous,
satisfying the same conditions as $B_0$.
On the other hand, if  ${\sigma}(B)(x,{\xi}) = {\xi}^j_2 B_j(x,{\xi})$
with $j \ge k$ then we can write
\begin{equation*}
 B(x,D_x) \cong  A(x,D_{x})R(D_x)\qquad \text{modulo ${\Psi}^0$}
\end{equation*}
for some $A \in {\Psi}^0$ and then
\begin{equation*}
 \begin{pmatrix}
1 & -A \\ 0 & 1 
\end{pmatrix}
P
 \begin{pmatrix}
1 & A \\ 0 & 1 
\end{pmatrix}
\cong
\begin{pmatrix}
 D_{x_1} & 0 \\ 0 & D_{x_1}  + R(D_x)
\end{pmatrix}
\qquad\text{modulo ${\Psi}^0$}
\end{equation*}
which is solvable. 
In fact, the principal symbol is on diagonal form 
with real diagonal elements of principal type, giving
$L^2$ estimates of the adjoint which can be perturbed by lower order terms. 
\end{exe}

Finally, we have an example of an unsolvable operator which is
diagonalizable and self-adjoint, but not of principal type.

\begin{exe} \label{nonsolvex}
Take real  $b(t) \in C^\infty(\br)$, and define the symmetric system
\begin{equation*}
P = 
\begin{pmatrix}
D_t + b(t)D_x & (t - ib(t))D_x  \\ (t + ib(t))D_x  & -D_t + b(t)D_x 
\end{pmatrix} = P^* 
\qquad (t,x) \in \br^2
\end{equation*}
Eigenvalues of ${\sigma}(P)$ are $b(t){\xi} \pm \sqrt{{\tau}^2 +
  (t^2 + b^2(t)){\xi}^2} $ which are zero for $({\tau},{\xi}) \ne 0$ 
only if $t = {\tau}= 0$. 
The eigenvalues coincide for $({\tau},{\xi}) \ne 0$ if and only if
$b(t) = t = {\tau}= 0$.  
We have that 
\begin{equation*}
 Q = \frac{1}{2}
\begin{pmatrix}
1 & -i \\ 1 & i 
\end{pmatrix}P
\begin{pmatrix}
1 & 1 \\ -i & i 
\end{pmatrix} =
\begin{pmatrix}
D_t - itD_x &  2b(t)D_x \\ 0 & D_t +itD_x
\end{pmatrix}
\end{equation*}
which is not locally  solvable at $t=0$ for any choice of~$b(t)$, since $D_t +itD_x$ is
not locally solvable, condition ~(${\Psi}$) is not satisfied when
${\xi} > 0$. The eigenvalues of the principal symbol ${\sigma}(Q)$
are ${\tau} \pm it{\xi}$. By the invariance, $P$ is of principal type if
and only if $b(0)=0$. When $b(t) \ne 0$ we find that
${\sigma}(P)$ is diagonalizable and self-adjoint, but not of principal
type. When $b \equiv 0$ the system is symmetric of principal type, but
does not have constant characteristics.
\end{exe}

\section{The multiplier estimates}\label{reduc}

In this section we shall prove multiplier estimates
for microlocal normal forms of the adjoint
operator, which we shall use in the proof of Theorem~\ref{mainthm}.  We shall consider the model operators 
\begin{equation}\label{pdef1}
P_0 = (D_t + i F(t,x,D_x))\id_N + F_0(t,x,D_x) 
\end{equation}
where $F \in C^\infty(\br, {\Psi}^1_{cl}(\br ^{n}))$ is scalar with 
with real homogeneous principal
symbol ${\sigma}(F) = f$, and $F_0 \in  C^\infty(\br, {\Psi}^0_{cl})$
is an $N \times N$ system. 
In the following, we shall
assume that ~~$P_0$ satisfies
condition ~($\overline{\Psi}$): 
\begin{equation} \label{pcond0}
f(t,x,{\xi}) > 0  \quad\text{and $s > t$} \implies  f(s,x,{\xi}) \ge 0
\end{equation}
for any ~$t$, $s \in \br$ and $(x,{\xi}) \in T^*\br^n$. This means
that the ~adjoint $P_0^*$ satisfies condition ~(${\Psi}$) for the
eigenvalue ${\tau} - if(t,x,{\xi})$. Observe
that if ${\chi} \ge 0$ then ${\chi}f$ also satisfies ~\eqref{pcond0}, thus
the condition can be localized.

\begin{rem} \label{linftyrem}
We may also consider symbols $f \in L^\infty(\br, S^1_{1,0}(\br
^{n}))$, that is, $f(t,x,{\xi}) \in L^\infty(\br \times T^*\br^n)$\/
is bounded in~$S^1_{1,0}(\br ^{n})$ for almost all ~$t$. Then we say
that $P_0$ satisfies condition ~($\ol {\Psi}$) if for every~$(x,{\xi})$
condition~\eqref{pcond0} holds for almost all $s$, $t \in \br$.
\end{rem}

Observe that, since 
$(x,{\xi}) \mapsto f(t,x,{\xi})$ is continuous for almost all ~$t$, it
suffices to check ~\eqref{pcond0} for~$(x,{\xi})$ in a countable dense
subset of~ $T^*\br^n$. Then we find that $f$ has a representative
satisfying~\eqref{pcond0} for any~$t$, $s$ and ~$(x,{\xi})$ after
putting $f(t,x,{\xi}) \equiv 0$ for ~$t$ in a null
set.

In order to prove Theorem~\ref{mainthm} we shall make a
second microlocalization using the specialized symbol classes of
the Weyl calculus, and the Weyl quantization of symbols $a \in \Cal
S'(T^*\br^n)$ defined by: 
\begin{equation*} 
\sw{a^wu,v} = (2{\pi})^{-n}\iint
\exp{(i\w{x-y,{\xi}})}a\!\left(\tfrac{x+y}{2},{\xi}\right)u(y)\ol
v(x)\,dxdyd{\xi}
\qquad u, v \in \Cal S(\br^n)
\end{equation*} 
Observe that $\re a^w = (\re a)^w$ is the symmetric
part and $i\im a^w = (i\im a)^w$ the antisymmetric part of the
operator $a^w$. Also, if $a \in S^m_{1,0}(\br ^n)$ then $a^w(x,D_x)
= a(x,D_x)$ modulo ${\Psi}^{m-1}_{1,0}(\br ^n)$ by \cite[Theorem~
18.5.10]{ho:yellow}. The same holds for $N \times N$ systems
of operators.

We recall the definitions of the Weyl calculus: let $g_{w}$ be a
Riemannean metric on~ $T^*\br^n$, $w = (x,{\xi})$, then we say that
$g$ is slowly varying if there exists $c>0$ so that $g_{w_0}(w-w_0) <
c $ implies
$$1/C \le g_{w}/g_{w_0}
\le C  $$ 
that is,  $g_{w} \cong g_{w_0}$.
Let ${\sigma}$ be the standard symplectic form on $T^*\br^n$, $ g^{\sigma}(w) $
the dual metric of $w \mapsto
g({\sigma}(w))$ and assume that $ g^{\sigma}(w) \ge g(w)$.  We say that $g$
is ${\sigma}$~ temperate if it is slowly varying and
\begin{equation*}
 g_{w} \le C g_{w_0}( 1 +
 g^{\sigma}_{w}(w-w_0))^N  \qquad
\text{$w$, $w_0 \in T^*\br^n$} 
\end{equation*}
A positive real valued function $m(w)$ on $T^*\br^n$ is 
$g$~continuous if there exists $c>0$ so that
$g_{w_0}(w-w_0) < c $ implies $m(w) \cong
m(w_0)$.
We say that $m$ is 
${\sigma}$, $g$~ temperate if it is $g$ ~continuous and 
\begin{equation*}
 m(w) \le C m(w_0)( 1 +
 g^{\sigma}_{w}(w-w_0))^N  \qquad\text{$w$,
   $w_0 \in T^*\br^n$}  
\end{equation*}
If $m$ is ${\sigma}$, $g$~ temperate, 
then $m$ is a weight for ~$g$ and we can define the symbol classes:
$a \in S(m,g)$ if $a \in C^\infty(T^*\br^n)$ and
\begin{equation}\label{symbolest} 
|a|^g_j(w) = \sup_{T_i\ne 0}
\frac{|a^{(j)}(w,T_1,\dots,T_j)|}{\prod_1^j
g_{w}(T_i)^{1/2}}\le C_j
m(w)\qquad w \in T^*\br^n \qquad\text{$j \ge 0$}
\end{equation}
which defines the seminorms of $S(m,g)$. Of course, these symbol
classes can also defined locally.
For matrix valued symbols, we use the matrix norms.
If $a \in
S(m,g)$ then we say that the corresponding Weyl operator $a^w \in \op
S(m,g)$.  
For more results on the Weyl calculus, see \cite[Section
18.5]{ho:yellow}.

\begin{defn}\label{s+def}
Let $m$ be a weight for the metric $g$.  We say that $a \in S^+(m,g)$
if $a \in C^\infty(T^*\br^n)$ and $|a|_j^g \le C_jm$ for $j \ge
1$.
\end{defn}

Observe that if $a \in S^+(m,g)$ then 
$a$ is a symbol. In fact, since $g \le g^{\sigma}$ we find by
integration that 
\begin{equation*}
\begin{split} 
 |a(w) - a(w_0)| \le C_1\sup_{{\theta} \in [0,1]} m(w_{\theta}) 
 g_{w_{\theta}}(w-w_0)^{1/2} \le C_N m(w_0)(1 +
 g_{w_0}^{{\sigma}}(w-w_0))^{N_0}  
\end{split}
\end{equation*}
where $w_{\theta} = {\theta}w + (1-{\theta})w_0$,
which implies that $m + |a|$ is a weight for ~$g$.
Clearly, $a \in S(m + |a|,
g)$, so the operator ~$a^w$ is well-defined.

\begin{lem}\label{calcrem}
Assume that $m_j$ is a weight for for the ${\sigma}$ temperate conformal metrics
$g_j = h_j g^\sh \le g^\sh = (g^\sh)^{\sigma}$ and $a_j \in  
S^+(m_j, g_j)$, $j=1$, $2$.  Let $g = (g_1 + g_2)/2$ and $h^2 = \sup
g_1/g_2^{\sigma} = \sup g_2/g_1^{\sigma}$, then we find that $h^2  =
h_1h_2$ and
\begin{equation}\label{2.3}
a_1^wa_2^w -(a_1a_2)^w \in \op S(m_1m_2h,g)
\end{equation}
We also obtain the usual expansion of~\eqref{2.3}  with terms in $
S(m_1m_2h^k,g)$, $k\ge 1$. 
\end{lem}

Observe that by Proposition~18.5.7  and~(18.5.14) in ~\cite{ho:yellow} we find that $g$
is ${\sigma}$ temperate and $g/g^{\sigma} \le (h_1 + h_2)^2/4 \le 1$.

\begin{proof} 
As showed after Definition~\ref{s+def} we have that $m_j + |a_j|$ is a
weight for ~$g_j$ and $a_j \in S(m_j + |a_j|, g_j)$, $j = 1$, 2.
Thus 
$$a_1^w a_2^w \in \op S((m_1 + |a_1|)(m_2 + |a_2|),g)$$
is given by  Proposition~18.5.5 in~\cite{ho:yellow}.
We find that $a_1^wa_2^w -(a_1a_2)^w = a^w$ with
\begin{equation*}
a(w) = E(\tfrac{i}{2}{\sigma}(D_{w_1},D_{w_2}))
\tfrac{i}{2}{\sigma}(D_{w_1},D_{w_2})a_1(w_1)a_2(w_2)\restr{w_1= w_2=w}  
\end{equation*}
where $E(z) = (e^z -1)/z = \int_0^1 e^{{\theta}z}\,d{\theta}$. 
We have that
${\sigma}(D_{w_1},D_{w_2})a_1(w_1)a_2(w_2) \in S(M,G)$ where
$$
M(w_1,w_2) = m_1(w_1)m_2(w_2)h_1^{1/2}(w_1)h_2^{1/2}(w_2) 
$$
and $G_{w_1,w_2}(z_1,z_2) = g_{1,w_1}(z_1) + g_{2,w_2}(z_2)$. Now the
proof of Theorem~18.5.5 in ~\cite{ho:yellow} works also when
${\sigma}(D_{w_1},D_{w_2})$ is replaced by
${\theta}{\sigma}(D_{w_1},D_{w_2})$, uniformly in $0 \le {\theta} \le
1$. By using Proposition~18.5.7 in ~\cite{ho:yellow} and integrating over
${\theta} \in [0,1]$ we obtain that $a(w)$ has 
an asymptotic expansion in $S(m_1m_2 h^k,g)$, which proves the Lemma.
\end{proof}

\begin{rem}\label{vvcalc}
The conclusions of Lemma~\ref{calcrem} also hold if $a_1
$ has values in $\Cal L(B_1, B_2)$ and $a_2$ has values in~ $B_1$ 
where $B_1$ and $B_2$ are Banach spaces (see Section~18.6
in~\cite{ho:yellow}).
\end{rem}
 
For example, if $\set{a_j}_j \in S(m_1,g_1)$ with values in $\ell^2$,
and $b_j \in S(m_2,g_2)$ uniformly in ~$j$, then $\set{a_j^wb_j^w}_j \in
\op (m_1m_2,g)$ with values in ~$\ell^2$.  
Thus, if $\set{{\phi}_j}_j \in S(1,g)$ is a partition of unity so
that $\sum_{j}{\phi}_j^2 = 1$ and $a \in S(m,g)$, then
$\set{{\phi}_ja}_j \in S(m,g)$ has values in $\ell^2$.

\begin{exe}
The standard symbol class $S^{\mu}_{{\varrho},{\delta}}$ defined by 
\begin{equation*}
 |\partial_x^{\alpha}\partial_{\xi}^{\beta} a(x,{\xi})| \le
 C_{{\alpha}{\beta}}\w{{\xi}}^{{\mu} + {\delta}|{\alpha}| -
   {\varrho}|{\beta}|} 
\end{equation*}
has ${\sigma}$~temperate metric if $0 \le {\delta}
\le {\varrho} \le 1$ and ${\delta} < 1$.  
\end{exe}

In the proof of
Theorem~\ref{mainthm} we shall microlocalize near~$(x_0,{\xi}_0)$ and
put $h^{-1} = \w{{\xi}_0} = 1 + |{\xi}_0|$. Then after doing a
symplectic dilation: $(x,{\xi}) \mapsto (h^{-1/2}x, h^{1/2}{\xi})$, we
find that $S^k_{1,0} = S(h^{-k}, h g^\sh)$ and $S^k_{1/2,1/2} =
S(h^{-k}, g^\sh)$, $k \in \br$, where $g^\sh = (g^\sh)^{\sigma}$ is
the Euclidean metric. We shall
prove a semiclassical estimate for a microlocal normal form
of the operator.

Let $\mn u$ be the $L^2$ norm on $\br^{n+1}$, and $\sw{u,v}$ the
corresponding sesquilinear inner product. As before, we say that $f
\in L^\infty(\br, S(m, g))$ if $f(t,x,{\xi})$\/ is measurable and
bounded in~$S(m, g)$ for almost all ~$t$.  The following is
the main estimate that we shall prove.

\begin{prop}\label{mainprop} 
Assume that 
$$P_0 = (D_t + i f^w(t,x,D_x))\id_N + F_0^w(t,x,D_x)$$ 
where $f \in
L^\infty(\br, S(h^{-1}, hg^\sh))$ is real satisfying
condition~$(\overline{\Psi})$ given by ~\eqref{pcond0},
and $F_0 \in L^\infty(\br, S(1, hg^\sh))$ is an $N \times N$ system, here $0 < h
\le 1$ and $g^\sh = (g^\sh)^{\sigma}$ are constant. Then there
exists\/ $T_0 > 0$ and $N \times N$ symbols $b_T(t,x,{\xi}) \in
L^\infty(\br, S(h^{-1/2}, g^\sh) \bigcap S^+(1, g^\sh))$ such that
$\im b_T \in L^\infty(\br, S(h^{1/2}, g^\sh))$
uniformly for $0 < T \le T_0$, and
\begin{equation}\label{propest}
 h^{1/2}\left(\mn{b_T^wu}^2 + \mn u^2\right) \le C_0 T \im\sw{P_0u,b_T^wu}
\end{equation} 
for $u(t,x) \in \Cal S(\br\times \br^{n}, \bc^N)$ having support where
$|t| \le T$. The constants $C_0$, $T_0$ and the seminorms of \/$b_T$ only depend
on the seminorms of $f$ and $F_0$.
\end{prop}

\begin{rem} 
It follows from the proof that $b_T = \wt b_TE^*E$ modulo $S(h^{1/2},
g^\sh)$ where $E \in S(1, hg^\sh)$ is an invertible
$N \times N$ system, $\wt b_T$ is scalar and $|\wt b_T| \le CH^{-1/2}$, here $H$ is
a weight for $g^\sh$ such that $h \le H \le 1$, and $G = Hg^\sh$
is ${\sigma}$ temperate (see Claim~\ref{subrem},
Definition~\ref{g1def} and Proposition~\ref{wickweyl}). 
\end{rem}

Observe that it follows from ~\eqref{propest}
and the Cauchy-Schwarz inequality that
\begin{equation*}
\mn u \le C  T h^{-1/2}\mn{P_0u}
\end{equation*}
which will give a loss of $3/2$ derivatives after microlocalization.
Proposition~\ref{mainprop} will be proved in Section
~\ref{lower}.

There are two difficulties present in estimates of the
type~\eqref{propest}. The first is that $b_T$ is not $C^\infty$ in the
$t$~variables, therefore one has to be careful not to involve $b_T^w$
in the calculus with symbols in all the variables. We shall avoid this
problem by using tensor products of operators and the Cauchy-Schwarz
inequality.  The second difficulty lies in the fact that we could have
$|b_T| \gg h^{1/2}$, so it is not obvious that cut-off errors can be
controlled.

\begin{lem}\label{microrem} 
The estimate ~\eqref{propest} can be perturbed with terms in
$L^\infty(\br, S(h^{1/2}, hg^\sh))$ in the expansion of $P_0$. Also,
it can be microlocalized: if ${\phi}(w) \in S(1, hg^\sh)$ is real
valued and independent of ~$t$, then we have
\begin{equation}\label{cutofferror}
\im\sw{P_0{\phi}^w u, b_{T}^w{\phi}^w u} \le
\im\sw{P_0u,{\phi}^w b_T^w{\phi}^w u} + Ch^{1/2}\mn u^2 \qquad u(t,x)
\in \Cal S(\br^{n+1}, \bc^N) 
\end{equation}
where ${\phi}^wb_T^w{\phi}^w$ satisfies the same conditions as~$b_T^w$.
\end{lem}

\begin{proof} 
In the following, we shall say that a system is {\em real\/} if it
is a real multiple of the identity matrix.
It is clear that we may perturb ~\eqref{propest} with
terms in $L^\infty(\br, S(h^{1/2}, g^\sh))$ in the expansion of ~$P_0$
for small enough~$T$. Now, we can also perturb with real
terms $r^w \in L^\infty(\br, \op S(1,h g^\sh))$. In fact, if $r
\in S(1, hg^\sh)$ is {real} and $B \in S^+(1, g^\sh)$ is symmetric
modulo $S(h^{1/2}, g^\sh)$, then
\begin{equation}\label{2.4}
 | \im\sw{r^w u, B^w u}| \le
 |\sw{[(\re B)^w,r^w]u,u}|/2  +
  |\sw{r^wu,(\im B)^wu}|  \le C  h^{1/2} \mn u^2
\end{equation}
In fact, we have $[(\re B)^w,r^w] \in \op S(h^{1/2},
g^\sh)$ by Lemma~\ref{calcrem}. 

If ${\phi}(w) \in S(1, hg^\sh)$ then $[P_0,{\phi}^w\id_N]
= \pb {f,{\phi}} ^w\id_N$ modulo $L^\infty(\br, \op S(h,
hg^\sh))$ where $\pb {f,{\phi}} \in L^\infty(\br, S(1, hg^\sh))$
is real valued.  By using~\eqref{2.4} with $r^w = \pb {f,{\phi}}^w\id_N$
and $B^w = b_T^w{\phi}^w$, we obtain ~\eqref{cutofferror} since
$b_T^w{\phi}^w \in \op S^+(1, g^\sh)$ is symmetric modulo $\op
S(h^{1/2}, g^\sh)$ for almost all ~~$t$ by Lemma~\ref{calcrem}. 
Since Lemma~\ref{calcrem} also gives that ${\phi}^wb_T^w{\phi}^w =
{\phi}^w(b_T\phi)^w = (b_T{\phi}^2)^w$ modulo 
$L^\infty(\br, \op S(h, g^\sh))$ we find that ${\phi}^wb_T^w{\phi}^w$
satisfies the same conditions as~$b_T^w$.
\end{proof}

\begin{claim}\label{subrem}
When proving the estimate ~\eqref{propest} we may assume that
\begin{equation}\label{subform}
 F_0 =  \w{d_w f, R_0}= \sum_j \partial_{w_j}fR_{0,j} \qquad\text{modulo
   $L^\infty(\br,  S(h, hg^\sh))$} 
\end{equation}
where $R_{0,j} \in L^\infty(\br,  S(h^{1/2}, hg^\sh))$ are $N \times N$
systems, $\forall \, j$.
\end{claim}

\begin{proof} 
By conjugation with $(E^{\pm 1})^w \in \op S(1, hg^\sh)$ we find
that 
\begin{equation*}
 (E^{-1})^w P E^w = (E^{-1})^wE^w(D_t + if^w)\id_N + \left(E^{-1}(D_tE + H_fE +
   F_0E)\right)^w = \wt P
\end{equation*}
modulo $L^\infty(\br,  S(h, hg^\sh))$. By solving 
\begin{equation*}
 \left\{
\begin{aligned}
&D_t E + F_0E = 0\\
&E\restr{t=0} = \id_N
\end{aligned}
\right.
\end{equation*}
we obtain ~\eqref{subform} for
$\wt P$ with  $ \w{d_w f, R_0}= E^{-1}H_fE$. From the calculus we
obtain that
 $$
E^w(E^{-1})^w = 1 = (E^{-1})^wE^w\qquad\text{modulo $\op S(Th,hg^\sh)$}
 $$
uniformly when $|t| \le T$. Thus, for small enough ~$T$ we obtain that
$(E^{\pm 1})^w$ is invertible in ~$L^2$. Since
the metric $hg^\sh$ is trivially strongly ${\sigma}$ ~temperate in the
sense of \cite[Definition~7.1]{bc:sob}, we find from
\cite[Corollary~7.7]{bc:sob} that there exists $A \in L^\infty(\br,  S(1, hg^\sh))$
such that $E^w A^w = 1$. Thus, if we prove the
estimate~\eqref{propest} for $\wt P$ and 
substitute $u = A^wv$ we obtain
the estimate for $P$ with $b_T$ replaced by $((E^{-1})^w)^*b_T^wA^w$. Since $A = E^{-1}$
modulo $S(h, hg^\sh)$ we find from Lemma~\ref{calcrem} as before that the symbol of  this
multiplier is in $S(h^{-1/2}, g^\sh)\bigcap S^+(1,g^\sh)$ and that it is
symmetric modulo $ S(h^{1/2}, g^\sh)$. 
\end{proof}

We shall see
from the proof that if $F_0$ is on the form~\eqref{subform} then $b_T
= b_T\id_N$ is real.  Thus, in general the symbol of the multiplier will be on 
the form $b_T(E^{-1})^*E^{-1}$  modulo $S(h^{1/2}, g^\sh)$ with
invertible $E$ and a real scalar ~$b_T$.
In the following, we shall use the partial Sobolev
norms: 
\begin{equation} 
\mn u_{s} = \mn{\w{D_x}^{s}u}
\end{equation} 
We shall now prove the estimate we shall use in the proof of
Theorem~\ref{mainthm}.

\begin{prop}\label{maincor}
Assume that 
$$P_0 = (D_t + i F^w(t,x,D_x))\id_N + F_0^w(t,x,D_x)$$ 
with $F^w \in
L^\infty(\br, {\Psi}^1_{cl}(\br^{n}))$ having real principal symbol
$f$ satisfying condition~$(\overline{\Psi})$ given by
~\eqref{pcond0} and $F_0 \in L^\infty(\br, {\Psi}^0_{cl}(\br^{n}))$ is
an $N \times N$ system. 
Then there exists $T_0 > 0$ and $N \times N$ symbols
$B_T(t,x,{\xi}) \in L^\infty(\br, S^1_{1/2,1/2}(\br^{n}))$ with
$$\nabla B_T = (\partial_x B_T, |{{\xi}}| \partial_{\xi} B_T)
\in L^\infty(\br, S^{1}_{1/2,1/2}(\br^{n}))$$ and $\im B_T(t,x,{\xi})
\in L^\infty(\br, S^0_{1/2,1/2}(\br^{n}))$ uniformly for $0 < T \le
T_0$, such that
\begin{equation}\label{corest}
\mn{B_T^wu}^2_{-1/2}+ \mn u^2\le  C_0(T \im\sw{P_0u,B_T^wu} + 
 \mn{u}_{-1}^2)
\end{equation} 
for $ u \in \Cal S(\br^{n+1}, \bc^N)$ having support where $|t| \le T$.
The constants $T_0$, $C_0$ and the
seminorms of\/ $B_T$ only depend on the seminorms of $F$  and $F_0$ in
$L^\infty(\br,S^1_{cl}(\br^{n}))$.
 \end{prop}

Since $\nabla B_T  \in
L^\infty(\br, S^{1}_{1/2,1/2})$ we find that the commutators of
$B_T^w$ with scalar operators in $L^\infty(\br,{\Psi}^0_{1,0})$ are in
$L^\infty(\br, {\Psi}^{0}_{1/2,1/2})$. This will make it possible to
localize the estimate. The idea to include the first term in ~\eqref{corest} is due to Lerner
~\cite{ln:cutloss}.

\begin{proof}[Proof that Proposition~\ref{mainprop} gives
  Proposition~\ref{maincor}] 
Choose real symbols $\set{{\phi}_j(x,{\xi})}_{j}$ 
and $\set{{\psi}_j(x,{\xi})}_{j} \in
S^0_{1,0}(\br^{n})$ having values in ~$\ell^2$, such that $\sum_{j}
{\phi}_j^2 = 1$, ${\psi}_j{\phi}_j = {\phi}_j$ and ${\psi}_j \ge 0$.
We may assume that the supports are 
small enough so that $\w {\xi} \cong \w {{\xi}_j}$ in $\supp {\psi}_j$
for some~ ${\xi}_j$, and that there is a fixed bound on number of
overlapping supports. Then, after doing a symplectic dilation
$$(y,{\eta}) = (x\w{{\xi}_j}^{1/2}, {\xi}/\w{{\xi}_j}^{1/2})$$ 
we obtain that $S^m_{1,0}(\br^n) = S(h_j^{-m}, h_j g^\sh)$ and $S^m_{1/2,1/2}
(\br^n) = S(h_j^{-m}, g^\sh)$ in $\supp {\psi}_j$, $m \in \br$, where
$h_j = \w {{\xi}_j}^{-1}\le 1$ and $g^\sh(dy, d{\eta}) = |dy|^2 +
|d{\eta}|^2$ is constant.

By using the calculus in the $y$ variables we find $
{\phi}_j^wP_0 = {\phi}_j^w P_{0j}$ modulo $\op S(h_j, h_jg^\sh)$, where
\begin{multline} 
P_{0j} = (D_t + i({\psi}_jF)^w(t,y,D_y))\id_N +
({\psi}_jF_0)^w(t,y,D_y)\\ = (D_t + if_j^w(t,y,D_y))\id_N +
F_j^w(t,y,D_y)
\end{multline}
with $f_j = {\psi}_j f \in L^\infty(\br,S(h_j^{-1}, h_j g^\sh)) $
satisfying~\eqref{pcond0}, and $F_j \in L^\infty(\br, S(1,h_j g^\sh))$
uniformly in ~$j$.  Then, by using Proposition~\ref{mainprop} and
Lemma~\ref{microrem} for $P_{0j}$ we obtain symbols
$b_{j,T}(t,y,{\eta}) \in L^\infty(\br, S(h_j^{-1/2}, g^\sh) \bigcap
S^+(1, g^\sh))$ such that 
$\im b_{j,T} \in S(h_j^{1/2}, g^\sh)$ uniformly for $0 < T \ll 1$, and
\begin{equation}\label{mmest}
\mn {b_{j,T}^w{\phi}_j^w u}^2 + \mn {{\phi}_j^w u}^2 \le C_0T (h_j^{-1/2}\im\sw{P_{0}u,
  {\phi}_j^w b_{j,T}^w {\phi}_j^wu} + \mn u^2)\qquad \forall\,j
\end{equation} 
for $u(t,y) \in \Cal S(\br\times \br^{n}, \bc^N)$ having support where
$|t| \le T$. Here and in the following, the constants are independent
of~$T$.

By substituting ${\psi}_j^wu$ in~\eqref{mmest} 
and summing up we obtain
\begin{equation}\label{2.12}
\mn{B_T^wu}^2_{-1/2} + \mn {u}^2 \le C_0T(\im \sw{P_0u, B_{T}^wu} + \mn
u^2) + C_1\mn{u}_{-1}^2
\end{equation}
for $u(t,x) \in \Cal S(\br\times \br^{n}, \bc^N)$ having support where
$|t| \le T$.
Here 
$$B_{T}^w = \sum_j h_j^{-1/2} {\psi}_j^w {\phi}_j^w
b_{j,T}^w {\phi}_j^w{\psi}^w_j = \sum_{j} B_{j,T}^w\in L^\infty(\br,
{\Psi}^{1}_{1/2,1/2})
$$ so $\im B_T \in  L^\infty(\br, {\Psi}^{0}_{1/2,1/2})$.  In fact,
since $d{\psi}_j = 0 $ on $\supp {\phi}_j$ we have
$$\set{{\phi}_j^w[P_{0j}^w,{\psi}_j^w]}_j \in
{\Psi}^{-1}_{1,0}(\br^n)$$ 
with values in $\ell^2$ for almost all ~$t$. Also,
$\sum_{j}^{}{\phi}_j^2 = 1$ so $\sum_j 
{\phi}_j^w{\phi}_j^w = 1$ modulo ${\Psi}^{-1}(\br^n)$, and by the 
finite overlap of supports we find that 
\begin{multline*}
 (\w{D_x}^{-1/2}B_T^w)^*\w{D_x}^{-1/2}B_T^w =
 (B_T^w)^*\w{D_x}^{-1}B_T^w \\ = \sum_{|j-k| \le K}(B_{j,T}^w)^*
 \w{D_x}^{-1}B_{k,T}^w \qquad \text{modulo ${\Psi}^{-2}$} 
\end{multline*}
for some $N$, which implies that
\begin{equation*}
 \mn{B_T^wu}^2_{-1/2} \le C_K \left(\sum_k  \mn{B_{k,T}^w u}_{-1/2}^2  + \mn u^2_{(-1)}\right)
\end{equation*}
We also have that
$\w{D_x}^{-1/2}h_j^{-1/2}{\psi}^w_j{\phi}^w_j \in {\Psi}^{0}(\br^n)$
uniformly, which gives
\begin{equation*}
 \mn{B_{k,T}^w u}_{-1/2} \le C \mn{b_{k,T}^w
   {\phi}^w_k{\psi}^w_ku} \qquad \forall\,k
\end{equation*}
We find that $\nabla B_{T} \in S^{1}_{1/2,1/2}$ since
Lemma~\ref{calcrem} gives
$$B_{T} =
\sum_{j}^{}h_j^{-1/2}b_{j,T}{\phi}_j^2 \in S^{1}_{1/2,1/2}\qquad \text{modulo
$S^{0}_{1/2,1/2}$}
$$ 
where ${\phi}_j \in S(1, h_jg^\sh)$ and $b_{j,T}
\in S^+(1,g^\sh)$ for almost all~$t$. For small enough ~$T$ we obtain 
~\eqref{corest} and the corollary.
\end{proof}

\section{Proof of Theorem~\ref{mainthm}}\label{proof}

In order to prove the theorem, we first need a preparation result so that we can get the system on a
normal form.

\begin{prop}\label{prepprop}
Assume that $P \in S^m_{cl}(M)$ is a $N \times N$ system of principal
type having constant characteristics near $(x_0,{\xi}_0) \in T^*M$, then
there exist  
elliptic $N \times N$ systems $A$ and $B \in  S^0_{cl}(M)$ such that 
\begin{equation*}
 A^w P^w B^w = Q^w =
\begin{pmatrix}
Q_{11}^w & 0 \\
0 & Q_{22}^w 
\end{pmatrix} \in {\Psi}^{m}_{cl}
\end{equation*}
microlocally  near $(x_0,{\xi}_0)$. We have that ${\sigma}(Q_{11}) =
{\lambda}\id_K$ where the section of eigenvalues ${\lambda}(w)\in C^\infty$ of ~$P(w)$
is of principal type, and $Q_{22}^w$ is elliptic. 
\end{prop}

Thus we obtain the system on a block form. Observe that if
$K=0$ then $P$ is elliptic at $(x_0,{\xi}_0)$. Since $P$ is of principal
type we find by the invariance given by ~\eqref{invcomp} that $Q$
is of  principal type, so  ${\lambda}$ 
vanishes of first order on its zeros.

\begin{proof}
Since $P_m$ has constant
characteristics by the assumptions, we find that the characteristic equation
\begin{equation*}
 |P_m(w) - {\lambda}\id_N| = 0
\end{equation*}
has a unique local solution ${\lambda}(w) \in C^\infty$ of
multiplicity $K > 0$. Since $P_m(w)$ is of principal type,
Proposition~\ref{princprop} gives that $d{\lambda}(w_0) \ne 0$ and the
geometric multiplicity $\dim \ker 
(P_m(w) - {\lambda}(w)\id_N) \equiv K$ in a neighborhood of $w_0 =
(x_0,{\xi}_0)$. Since the dimension is constant,
we may choose a $C^\infty$ base for $ \ker
(P_m(w) - {\lambda}(w)\id_N) $ in a neighborhood of $w_0$. By
orthogonalizing it, extending to a orthonormal $C^\infty$ base for
$\bc^N$ and using homogeneity we obtain orthogonal homogeneous $E$ such that 
\begin{equation*}
 E^*P_mE = 
\begin{pmatrix}
{\lambda}(w)\id_K & P_{12} \\
0 & P_{22} 
\end{pmatrix} = \wt P_m = {\sigma}((E^w)^*P^wE^w)
\end{equation*}
Clearly $\ker\wt P_m = \set{(z_1,\dots,z_N):\ z_j = 0\text{ for } j > K}$ when
${\lambda}=0$ and $d\wt P_m$ is equal to multiplication with $ d{\lambda}$
on $\ker \wt P_m$. Since $\wt P_m$ is of principal type when ${\lambda}=0$ we find that $\im
\wt P_m \bigcap \ker \wt P_m = \set{0}$ at $w_0$, which implies that
$P_{22}$ is invertible. In fact, 
if it was not invertible there would exists $0 \ne z'' \in \bc^{N-K}$ so that
$P_{22}z'' = 0$, then 
$$0 \ne \wt P_m {}^{t}(0,z'') = {}^{t}(P_{12}z'', 0) \in \im \wt P_m
\bigcap \ker \wt P_m$$ 
giving a contradiction. By multiplying $\wt P_m$ from left with 
\begin{equation*}
 \begin{pmatrix}
 \id_K & -P_{12}P_{22}^{-1} \\
0 & \id_{N-K}
\end{pmatrix}
\end{equation*}
we obtain $P_{12} \equiv 0$. Thus, we find that 
\begin{equation*}
 A^w P^w B^w = 
\begin{pmatrix}
Q_{11}^w & Q_{12}^w \\
Q_{21}^w & Q_{22}^w 
\end{pmatrix} \in {\Psi}^{1}_{cl}
\end{equation*}
where ${\sigma}(Q_{11}) = {\lambda}\id_K$, $|{\sigma}(Q_{22})| \ne 0$
and $ Q_{12}$, $ Q_{21} \in {\Psi}^0_{cl}$. Choose a microlocal
parametrix $B_{22}^w \in {\Psi}^{-m}_{cl}$ to $Q_{22}^w$ so that
$B_{22}^w Q_{22}^w = Q_{22}^wB_{22}^w = \id_{N-K}$ modulo $C^\infty$ near ~$w_0$.
By multiplying from the left with
\begin{equation*}
\begin{pmatrix} 
\id_K & - Q_{12}^wB_{22}^w \\
0 & \id_{N-K}
\end{pmatrix}\in {\Psi}^{0}_{cl}
\end{equation*}
we obtain that $Q_{12} \in S^{-\infty}$. By multiplying from the right
with 
\begin{equation*}
\begin{pmatrix} 
\id_K & 0 \\
 -B_{22}^w Q_{21}^w & \id_{N-K}
\end{pmatrix} \in {\Psi}^{0}_{cl}
\end{equation*}
we obtain  $Q_{21} \in S^{-\infty}$. Note that these
multiplications do not change the principal symbols of $Q_{jj}$ for
$j=1$, 2, which finishes the proof.  
\end{proof}

\begin{proof} [Proof of Theorem~\ref{mainthm}] 
Observe that since $P$ satisfies condition ~(${\Psi}$) we find that
the adjoint $P^*$ satisfies condition ~($\ol {\Psi}$).
By multiplying with an elliptic pseudodifferential operator, we may
assume that $m=1$.  Let $P^*$ have the expansion $P_1 + P_0 + \dots$
where $P_1 = {\sigma}(P^*) \in S^1$, then it is clear that it
suffices to consider $w_0 = (x_0,{\xi}_0) \in |P_1|^{-1}(0)$, otherwise
$P^* \in {\Psi}^1_{cl}(M)$ is elliptic near ~$w_0$ so
~\eqref{solvest1} holds and $P$ is microlocally solvable. 
Now  ~$P^*$ is of principal type having constant characteristics so
we find by using Proposition~\ref{prepprop} that 
\begin{equation*}
 P^* = \begin{pmatrix}
Q_{11}^w & 0 \\
0 & Q_{22}^w 
\end{pmatrix} \in {\Psi}^{1}_{cl}
\end{equation*}
microlocally  near $w_0$, where ${\sigma}(Q_{11}) =
{\lambda}\id_K$ with ${\lambda} \in C^\infty$ an
eigenvalue of ${\sigma}(P^*)$ of principal type and $Q_{22}^w$ is elliptic.
Since $Q_{22}^w$ is
elliptic, it is trivially solvable so we only have to investigate the
solvability of $Q_{11}^w$. Now ${\lambda}$ is of principal type by the
invariance, so if it does not satisfy condition ~($\ol {\Psi}$) then the
proof of ~\cite[Theorem~26.4.7]{ho:yellow} can easily be
adapted to this case, since the principal part of the operator is a
scalar symbol times the identity matrix. 

To prove solvability when  condition ~($\ol {\Psi}$) is satisfied,
we shall prove that there exists ${\phi}$ and
${\psi} \in S^0_{1,0}(T^*M)$ such that ${\phi}= 1$ in a conical
neighborhood of $(x_0,{\xi}_0)$ and
for any ~$T >0$ there exists a $K \times N$ system $R_{T} \in S^{1/2}_{1,0}(M)$ with the
property that $\wf (R_{T}^w) \bigcap T^*_{x_0}M = \emptyset$ and
\begin{equation}\label{solvest}
\mn{{\phi}^wu} \le C_1\left(\mn{{\psi}^wP^*u}_{(3/2
-1)} + T\mn u\right) + \mn{R_{T}^w u} + C_0\mn{u}_{(-1)}
\qquad u \in C_0^\infty(M, \bc^N)
\end{equation}
Here $\mn{u}_{(s)}$ is the $L^2$ Sobolev norm and
the constants are independent of ~$T$.  Then for small
enough ~$T$ we obtain ~\eqref{solvest1} and microlocal solvability, 
since $(x_0,{\xi}_0) \notin \wf (1-{\phi})^w$.  In the case the
eigenvalue satisfies condition (${\Psi}$) and~\eqref{nontrap} near $x_0$ we
may choose finitely many 
${\phi}_j \in S^0_{1,0}(M)$ such that $\sum {\phi}_j \ge 1$ near~$x_0$
and $\mn{{\phi}_j^wu}$ can be estimated by the right hand side
of~\eqref{solvest} for some suitable ~${\psi}$ and ~$R_T$. By elliptic
regularity of $\set{{\phi}_j}$ near~$x_0$, we then obtain 
the estimate~\eqref{solvest1} for small
enough ~$T$ with $x \ne x_0$ in $\wf (R)$.

Observe that in the case when ${\lambda}$
satisfies condition  ($P$) we obtain the estimate~\eqref{solvest} for
$P^* = {\lambda}(x,D_x)\id_N$ with
$3/2$ replaced with $1$ and $C_1 = \Cal O(T)$ from the Beals-Fefferman estimate,
see~\cite{bf}. Since this estimate can be perturbed with terms in
${\Psi}^{0}_{cl}$ for small enough ~$T$ we get the estimate and solvability in this
case. A similar argument gives subelliptic estimates 
if ${\lambda}$ satisfies condition ~($\ol {\Psi}$) and the bracket
condition, see~\cite[Chapter~27]{ho:yellow}. This gives Remark~\ref{prem}.

It remains to consider the case $P_1 = {\lambda}\id_N$,
where ${\lambda}$
satisfies condition ~($\ol {\Psi}$). It is clear that by multiplying with
an elliptic factor
we may assume that $\partial_{\xi}\re {\lambda}(w_0) \ne 0$, in the
microlocal case after a conical transformation.  Then, we may use
Darboux' theorem and the Malgrange preparation theorem to obtain
microlocal coordinates $(t,y;{\tau},{\eta}) \in T^* \br^{n+1}$ so that
$w_0 = (0,0; 0,{\eta}_0)$, $t = 0$ on $T^*_{x_0}M$ and ${\lambda} =
q({\tau} + if) $ in a conical neighborhood of ~$ w_0$, where $f \in
C^\infty(\br,S^1_{1,0})$ is real and homogeneous satisfying
condition~\eqref{pcond0}, and $0 \ne q \in S^0_{1,0}$, see
Theorem~21.3.6 in ~\cite{ho:yellow}. By using the
Malgrange preparation theorem and homogeneity we find that 
\begin{equation*}
 P_0(t,x;{\tau},{\xi}) = Q_{-1}(t,x;{\tau},{\xi})({\tau} + if(t,x,{\xi}))\id_N + F_0(t,x,{\xi})
\end{equation*}
where $Q_{-1}$ is homogeneous of degree ~$-1$
and $F_0$ is homogeneous of degree ~$0$ in the ${\xi}$ ~variables. 
By conjugation with elliptic
Fourier integral operators and using the Malgrange preparation theorem
successively on lower order terms, we obtain that
\begin{equation} \label{mikrop}
P^* =  Q^w(D_{t}\id_N + i\left({\chi}F\right)^w) + R^w
\end{equation}
microlocally in a conical neighborhood ~${\Gamma}$ of ~$w_0$ as in
 the proof of Theorem~26.4.7$'$ in ~\cite{ho:yellow}. 
Here we find that $F \in C^\infty(\br,
S^{1}_{1,0 }(\br^{n}))$ has real principal symbol $f\id_N$
satisfying~\eqref{pcond0}, $Q \in S^0_{1,0}(\br^{n+1})$  has principal
symbol $q\id_N \ne 0$ in ~${\Gamma}$ and $R \in S^1_{1,0}( \br^{n+1})$
satisfies ${\Gamma} \bigcap \wf (R^w) = \emptyset$. Also,
${\chi}({\tau},{\eta}) \in S^0_{1,0}(\br^{n+1})$ is equal to 1 in
~${\Gamma}$ and $|{\tau}| \le C |{\eta}|$ in $\supp
{\chi}({\tau},{\eta})$.  By cutting off in the $t$~ variable
we may assume that $F \in L^\infty(\br, S^1_{1,0}(\br^{n}))$.
Now, we can follow the proof of Theorem~1.4 in ~\cite{de:cut}.
As before, we shall choose ~${\phi}$ and ~${\psi}$ so
that ${\phi} = 1$ conical neighborhood of ~$ w_0$, ${\psi} = 1$ on
$\supp {\phi}$ and $ 
\supp {\psi} \subset {\Gamma}$. Also, we shall choose
 $$
{\phi}(t,y;{\tau},{\eta}) =
{\chi}_0(t,{\tau},{\eta}){\phi}_0(y,{\eta})
 $$ 
where ${\chi}_0(t,{\tau},{\eta}) \in S^0_{1,0}(\br^{n+1})$,
${\phi}_0(y,{\eta}) \in S^0_{1,0}(\br^{n})$, $t
\ne 0 $ in $\supp \partial_t{\chi}_0$, $|{\tau}| \le C |{\eta}|$
in $\supp  {\chi}_0$ and $|{\tau}| \cong |{\eta}|$
in $\supp \partial_{{\tau},{\eta}} {\chi}_0$. 
 
Since $|{\sigma}(Q)| \ne 0$ and $R =0$ on $\supp {\psi}$ it is no restriction to
assume that $Q \equiv \id_N$ and $R \equiv 0$ when proving the
estimate~\eqref{solvest}.  Now, by Theorem~18.1.35 in~\cite{ho:yellow}
we may compose $C^\infty(\br, {\Psi}^m_{1,0}(\br^{n}))$ with operators
in ${\Psi}^k_{1,0}(\br^{n+1})$ having symbols vanishing when $|{\tau}|
\ge c(1+ |{\eta}|)$, and we obtain the usual asymptotic expansion in
${\Psi}^{m+k-j}_{1,0}(\br^{n+1})$ for $j \ge 0$. Since $|{\tau}| \le
C|{\eta}|$ in $\supp {\phi}$ and ${\chi} = 1$ on $\supp {\psi}$,
it suffices to prove ~\eqref{solvest} for $P^* = D_t + i F^w$.

By using Proposition ~\ref{maincor} on ${\phi}^wu$,
we obtain that
\begin{multline}\label{fest}
\mn{B_T^w{\phi}^wu}^2_{-1/2} + \mn{{\phi}^w u}^2 \\ \le C_0
T \left(\im\sw{{\phi}^wP^*u, B_T^w{\phi}^wu}+
 \im\sw{[P^*, {\phi}^w\id_N]u, B_T^w{\phi}^wu}\right) + C_1\mn{{\phi}^w u}^2_{-1}
\end{multline}
where $B_T^w \in L^\infty(\br,{\Psi}^{1}_{1/2,1/2}(\br^n))$ is an
$N \times N$ system with $\nabla B_T \in L^\infty(\br, 
S^{1}_{1/2,1/2}(\br^n))$, and $\mn{u}_s = \mn {\w{D_y}^su}$ is
the partial Sobolev norm in the $y$ variables.
Since $|{\tau}| \le C |{\xi}|$ in $\supp {\phi}$ we find that 
$ \mn{{\phi}^w u}_{-1} \le C \mn{ u}_{(-1)}$
For any $u,\ v \in \Cal S(\br^n, \bc^N)$ we have that
\begin{equation}\label{helpest}
|\sw{v, B^w_T u} | = |\sw{\w{D_y}^{1/2}v, \w{D_y}^{-1/2}B^w_Tu}| \le
C(\mn v_{1/2}^2  +
\mn {B^w_T u}^2_{-1/2})
\end{equation}
where $\w{D_y} = 1 + |D_y|$. 
Now ${\phi}^w = {\phi}^w{\psi}^w$ modulo
${\Psi}^{-2}_{1,0}(\br^{n+1}) $, thus we find from ~\eqref{helpest}
that
\begin{equation}\label{f1est}
|\sw{{\phi}^wP^*u, B_T^w{\phi}^wu}| 
\le C(\mn{{\psi}^wP^*u}_{1/2}^2 + \mn u^{2} + \mn {B_T^w{\phi}^wu}^2_{-1/2})
\end{equation}
where the last term can be cancelled for small enough ~$T$ in
~\eqref{fest}.
We also have to estimate the commutator term $\im\sw{[P^*, {\phi}^w \id_N]u,
B_T^w{\phi}^wu}$ in~\eqref{fest}. We find 
$$[P^*, {\phi}^w\id_N]
= -(i\partial_t {\phi}^w - \set{f,{\phi}}^w)\id_N \in
{\Psi}^{0}_{1,0}(\br^{n+1})$$
modulo ${\Psi}^{-1}_{1,0}(\br^{n+1})$
by the expansion, where the error term can be estimated by ~\eqref{helpest}. 
Since ${\phi}= {\chi}_0{\phi}_0$ we
find that $\set{f, {\phi}} = {\phi}_0 \set{f,{\chi}_0} + {\chi}_0
\set{f,{\phi}_0}$, where $ {\phi}_0\set{f,{\chi}_0} = R_0 \in
S^{0}_{1,0}(\br^{n+1})$ is supported when $|{\tau}| \cong |{\eta}|$
and ${\psi}=1$. Now $({\tau} + if)^{-1} \in S^{-1}_{1,0}(\br^{n+1})$
when $|{\tau}| \cong |{\eta}|$, thus by
~\cite[Theorem~18.1.35]{ho:yellow} we find that $ R_0^w =
A_1^w{\psi}^wP^*\text{ modulo ${\Psi}^{-1}_{1,0}(\br^{n+1})$} $ where
$A_1 = R_0({\tau} + if)^{-1} \in S^{-1}_{1,0}(\br^{n+1})$.  As before,
we find from ~\eqref{helpest} that
 \begin{multline}\label{f2est} 
|\sw{R_0^wu,B_T^w {\phi}^wu}| \le
C(\mn{R_0^w u}^2_{1/2} + \mn 
{B_T^w{\phi}^w u}_{-1/2}^2) \\\le C_0(\mn{{\psi}^wP^*u}_{-1/2}^2 + \mn
{B_T^w{\phi}^w u}_{-1/2}^2 + \mn u^2_{-1/2})
 \end{multline}
and also
$$|\sw{\partial_t {\phi}^wu,  B_T^w{\phi}^wu}| 
\le \mn{R_1^wu}^2 + \mn {B_T^w{\phi}^w u}_{-1/2}^2 $$ 
where $R_1^w =\w{D_{y}}^{1/2}\partial_t {\phi}^w\in
{\Psi}^{1/2}_{1,0}(\br^{n+1})$, thus $t \ne 0$ in $\wf (R_1^w)$.

It only remains to estimate the term $\im \sw{(\set{f,{\phi}_0}{\chi}_0)^wu,
  B_T^w{\phi}^wu}$. Here $(\set{f,{\phi}_0}{\chi}_0)^w = 
\set{f,{\phi}_0}^w{\chi}_0^w$ and ${\phi}^w = 
{\phi}_0^w{\chi}_0^w$ modulo ${\Psi}^{-1}_{1,0}(\br^{n+1})$.
As in ~\eqref{helpest} we find
$$|\sw{R^wu, B_T^wv}| = |\sw{\w{D_y} R^wu, \w{D_y}^{-1} B^w_Tv}|
\le C(\mn u^2 + \mn v^2)$$ 
for $R \in S^{-1}_{1,0}(\br^{n+1})$,
thus we find
\begin{equation*}
|\im \sw{(\set{f,{\phi}_0}{\chi}_0)^wu, B_T^w{\phi}^wu} |
\le  |\im \sw{\set{f,{\phi}_0}^w{\chi}_0^wu,
  B_T^w{\phi}_0^w{\chi}_0^wu}| +  C\mn u^2.
\end{equation*}
The calculus gives $B_T^w{\phi}_0^w = (B_T{\phi}_0)^w$  and
$$2i\im \left((B_T{\phi}_0)^w\set{f,{\phi}_0}^w \right)= 
\set{B_T{\phi}_0,\set{f,{\phi}_0}}^w = 0$$ 
modulo $L^\infty(\br,{\Psi}^{0}_{1/2,1/2}(\br^n))$ since ~$\nabla
(B_T{\phi}_0) \in L^\infty(\br, S^{1}_{1/2,1/2}(\br^{n}))$ and
$\set{f, {\phi}_0}$ is real. 
Thus, we obtain 
 \begin{equation}\label{f3est} 
|\im \sw{\set{f,{\phi}_0}^w{\chi}_0^wu,
  B_T^w{\phi}_0^w{\chi}_0^wu}| \le C\mn {{\chi}_0^wu}^2 \le C'\mn u^2
  \end{equation} 
and the estimate~\eqref{solvest} for small enough~$T$, which
completes the proof of Theorem~\ref{mainthm}.
\end{proof}

\section{The symbol classes and weights}\label{symb}

In this section we shall define the symbol classes we shall
use. Assume that $f \in L^\infty(\br, S(h^{-1}, 
hg^\sh))$ is scalar and satisfies ~\eqref{pcond0}, here $0 < h \le 1$ and $g^\sh =
(g^\sh)^{\sigma}$ are constant. It is no restriction to change $h$ so
that $|f|^{g^\sh}_1  \le h^{-1/2}$, which we assume in
what follows.
The results shall be uniform in the usual sense, they will only depend on
the seminorms of $f$ in $L^\infty(\br, S(h^{-1}, hg^\sh))$.  Let
\begin{align}\label{xpmdef}
&X_+(t) = \set{w \in T^*\br^n : \exists\,s\le t,\ f(s,w) >0}\\
&X_-(t) = \set{w \in T^*\br^n : \exists\,s\ge
 t,\ f(s,w) < 0}.\label{xpmdef1} 
\end{align}
Clearly, $X_\pm(t)$ are open in $T^*\br^n$, 
$X_+(s) \subseteq X_+(t)$ and $X_-(s) \supseteq X_-(t)$
when $s \le t$.
By condition $(\ol{\Psi})$
we obtain that $X_-(t) \bigcap X_+(t) = \emptyset$ and $\pm f(t,w) \ge
0$ when  $w \in X_\pm(t)$, $\forall\, t$.
Let $X_0(t) = T^*\br^n\setminus \left(X_+(t)\bigcup
  X_-(t)\right)$ which is closed in
$T^*\br^n$. By the definition of $X_\pm(t)$ we have
$f(t,w) =0$ when  $w \in X_0(t)$. 
Let
\begin{equation}\label{d0def}
d_0(t_0,w_0) = \inf\set{g^\sh(w_0-z)^{1/2}:\ z \in X_0(t_0)}
\end{equation}
be is the $g^\sh$ distance in $T^*\br^n$ to $ X_0(t_0)$ for fixed $t_0$,
it is equal to $+\infty$ in the case that 
$X_0(t_0) = \emptyset$. By taking the infimum over ~$z$ we
find that $w \mapsto d_0(t,w)$ is Lipschitz continuous with respect to
$g^\sh$ for fixed~$t$ when $ d_0 < \infty$, i.e., 
\begin{equation*}
\sup_{w \ne z \in  T^*\br^n}|{\delta}_0(t,w)-
{\delta}_0(t,z)|/g^\sh(w-z)^{1/2} \le 1
\end{equation*} 

\begin{defn} \label{d0deforig}   
We define the signed  distance function ${\delta}_0(t,w)$ by
\begin{equation}\label{delta0def}
{\delta}_0 =  \sgn(f)\min(d_0,h^{-1/2})
\end{equation}
where $d_0$ is given by ~\eqref{d0def} and 
\begin{equation}\label{fsign}
\sgn (f)(t,w) = 
\left\{
\begin{alignedat}{2}
 \pm 1&, &\quad &w \in X_\pm(t)\\
 0&, &\quad &w \in X_0(t)
\end{alignedat}
\right.
\end{equation}
so that $\sgn(f)f \ge 0$.
\end{defn}

\begin{rem}
The signed distance function $w \mapsto {\delta}_0(t,w)$ given by
Definition~\ref{d0deforig} is Lip\-schitz continuous with respect to the
metric $g^\sh$ with Lipschitz constant equal to 1, $\forall\, t$. We
also find that 
$t \mapsto {\delta}_0(t,w)$ is non-decreasing, ${\delta}_0f \ge 0$,
$|{\delta}_0| \le h^{-1/2}$ and when $|{\delta}_0| < h^{-1/2}$ we find
that $|{\delta}_0| = d_0$ is given by 
~\eqref{d0def}.
\end{rem}

In fact, it suffices to show the Lipschitz continuity of $w \mapsto
{\delta}_0(t,w)$ on ~$\complement{X_0}(t)$, and then it follows from
the Lipschitz continuity of $w \mapsto
d_0(t,w)$ when $d_0 < \infty$. 
Clearly ${\delta}_0f \ge 0$, and since $X_+(t)$ is non-decreasing
and $X_-(t)$ is non-increasing when $t$ increases, we find that $ t
\mapsto {\delta}_0(t,w)$ is non-decreasing.

In the following, we shall treat $t$ as a parameter which we shall
suppress, and  we shall denote $f' = \partial_wf$ and $f'' =
\partial_w^2f$.
We shall also in the following assume that we have choosen
$g^\sh$ orthonormal coordinates so that $g^\sharp(w) = |w|^2$ and
$|f'| \le h^{-1/2}$.

\begin{defn}\label{g1def}
Let
\begin{equation}\label{h2def}
H^{-1/2} = 1 + |{\delta}_0| + \frac{|f'|}{|{f''}| + h^{1/4}|f'|^{1/2} +
  h^{1/2}}
\end{equation}
and $G = H g^\sh$.
\end{defn}

Observe that $\w{{\delta}_0} = 1 + |{\delta}_0| \le H^{-1/2}$ and
\begin{equation} \label{H1hest}
1 \le H^{-1/2}\le 1 + |{\delta}_0| + h^{-1/4} |f'|^{1/2}\le 3 h^{-1/2}
\end{equation}
since $|f'| \le h^{-1/2}$ and $|{\delta}_0| \le
h^{-1/2}$. This gives that $hg^\sh \le 3 G$.

\begin{defn} \label{Mdef}
Let
\begin{equation}\label{Mdef0}
M = |f| +|f'| H^{-1/2} + |{f''}| H^{-1} + h^{1/2}H^{-3/2}
\end{equation}
then we have that $h^{1/2} \le M \le C_3 h^{-1}$.
\end{defn}

The metric $G$ and weight $M$ have the following properties
according to  Proposition~3.7 in \cite{de:NT}.

\begin{prop}\label{g1prop}
We find that $H^{-1/2}$ is Lipschitz continuous,
$G$ is ${\sigma}$ ~temperate  such that $G =
H^2G^{\sigma}$ and
\begin{equation}\label{tempest}
H(w) \le C_0 H(w_0)(1 + G_{w_0}(w-w_0)) 
\end{equation}
We have that $M$ is a weight for $G$ such that $M \le C H^{-1}$, $f \in S(M,G)$
and
\begin{equation}\label{mtemp}
M(w) \le C_1 M(w_0)(1 + G_{w_0}(w-w_0))^{3/2}
\end{equation}
\end{prop}

Since $G \le g^\sh \le
G^{\sigma}$ we find that the conditions ~\eqref{tempest}
and~\eqref{mtemp} are stronger than the property of being
${\sigma}$~temperate (in fact, strongly ${\sigma}$ ~temperate in the
sense of \cite[Definition~7.1]{bc:sob}).
Note that $f \in S(M,Hg^\sh)$ for any choice of $H \ge
h$ in Definition~\ref{Mdef}. The following property of ~$G$ is the most
important for the proof.

\begin{prop}\label{ffactprop}
Let $H^{-1/2}$ be given by Definition~\ref{g1def} for $f \in
S(h^{-1},hg^\sh)$. There exists ${\kappa_1}>0$ so that if\/
$\w{{\delta}_0} = 1 + |{\delta}_0| \le {\kappa}_1H^{-1/2}$ then
\begin{equation}\label{ffactor}
f = {\alpha}_0{\delta}_0
\end{equation}
where ${\kappa}_1MH^{1/2} \le {\alpha}_0 \in S(MH^{1/2}, G)$,
which implies that 
${\delta}_0 = f/{\alpha}_0 \in S(H^{-1/2}, G)$.
\end{prop}

This follows directly from Proposition~3.9 in~\cite{de:NT}.
Next, we shall define the weight ~$m$ we shall use. 

\begin{defn}\label{h0def}
For  $(t,w) \in \br\times T^*\br^n$ we
let 
\begin{multline}\label{mphodef}
  m(t,w) = \inf_{t_1 \le t \le
    t_2}\big\{\, |{\delta}_0(t_1,w) -{\delta}_0(t_2,w)|\\ +
    \max\big(H^{1/2}(t_1,w)\w{ {\delta}_0(t_1,w)}^2,
    H^{1/2}(t_2,w)\w{ {\delta}_0(t_2,w)}^2\big)/2 \,\big\}
\end{multline}
where $\w{{\delta}_0} = 1 + |{\delta}_0|$. 
\end{defn}

This weight essentially measures how much $t \mapsto
{\delta}_0(t,w)$ changes between the minima of $t \mapsto
H^{1/2}(t,w)\w{{\delta}_0(t,w)}^2$, which will give restrictions on
the sign changes of the symbol. 
When $t \mapsto {\delta}_0(t,w)$ is constant for fixed ~$w$, we find
that $t \mapsto m(t,w)$ is equal to the largest quasi-convex
minorant of $t \mapsto H^{1/2}(t,w)\w{{\delta}_0(t,w)}^2/2$, i.e., $\sup_I
m = \sup_{\partial I} m $ for compact intervals $I \subset \br$,
see~\cite[Definition~1.6.3]{ho:conv}.

The main difference between this weight and the weight in~\cite{de:NT} is
the use of $H^{1/2}\w{{\delta}_0}^2$ in the definition of
$m$ instead of $H^{1/2}\w{{\delta}_0}$, and this is due to
Lerner~\cite{ln:cutloss}. 
The weight has the following properties according to 
Propositions~4.3 and~4.4 in ~\cite{de:cut}.

\begin{prop}\label{h0slow}
We have that $m \in L^\infty(\br \times T^*\br^{n})$, 
$w \mapsto m(t,w)$ is uniformly Lipschitz continuous, $\forall\, t$,
and
\begin{equation} \label{hhhest} 
 h^{1/2}\w{ {\delta}_0}^2/6 \le
m \le H^{1/2}\w{ {\delta}_0}^2/2 \le \w{ {\delta}_0}/2
\end{equation}  
There exists $C>0$ so that
\begin{equation}\label{wtemp}
m(t_0,w) \le C m(t_0,w_0)(1 + |w-w_0|/\w{{\delta}_0(t_0,w_0)})^{3}
\end{equation}
thus $m$ is a weight for $g^\sh$.
\end{prop}

The following result will be essential for
the proof of Proposition~\ref{mainprop} in Section~\ref{lower},
it follows from Proposition~4.5  in~\cite{de:cut}.

\begin{prop}\label{mestprop}
Let the weight $M$ be given by Definition~\ref{Mdef} and $m$ by
Definition~\ref{h0def}. Then there exists $C_0>0$ such that
\begin{equation}\label{Mest0}
MH^{3/2}\w{{\delta}_0}^2 \le C_0m
\end{equation} 
\end{prop}

We have the following convexity property of $t \mapsto
m(t,w)$, which will be important for the construction of the 
multiplier. 

\begin{prop}\label{qmaxpropo}
Let $m$ be given by Definition~\ref{h0def}. Then
\begin{equation}\label{qmaxwd}
\sup_{t_1 \le t \le t_2}m(t,w) \le {\delta}_0(t_2,w) -
{\delta}_0(t_1,w) + m(t_1,w) + m(t_2,w)\qquad \forall\,w
\end{equation} 
\end{prop}

\begin{proof}
Since $t \mapsto {\delta}_0(t,w)$ is monotone, we find that
\begin{equation}\label{convest}
\inf_{\pm(t-t_0) \ge 0} \left(|{\delta}_0(t,w)
-{\delta}_0(t_0,w)| + H^{1/2}(t,w)\w{{\delta}_0(t,w)}^2/2\right) \le
m(t_0,w)
\end{equation}
Let $t \in [t_1,t_2]$, then by using ~\eqref{convest}  for $t_0 = t_1$,
$t_2$, and taking the infima, we obtain that
\begin{multline*} 
m (t,w) \le \inf_{r \le t_1 < t_2\le s} {\delta}_0(s,w)
-{\delta}_0(r,w) + H^{1/2}(s,w)\w{{\delta}_0(s,w)}^2/2  +
H^{1/2}(r,w)\w{{\delta}_0(r,w)}^2/2  \\
\le {\delta}_0(t_2,w) -{\delta}_0(t_1,w) +m (t_1,w) +m (t_2,w)
\end{multline*}
which gives~\eqref{qmaxwd} after taking the supremum.
\end{proof}

Next, we shall construct the pseudo-sign $B = {\delta}_0 + {\varrho}_0$,
which we shall use in  Proposition~\ref{wickweyl} to construct
the multiplier of Proposition~\ref{mainprop}.

\begin{prop}\label{apsdef}
Assume that ${\delta}_0$ is given by
Definition~\ref{d0deforig} and  $m$ is given by
Definition~\ref{h0def}. Then for $T>0$ there exists real 
valued ${\varrho}_T(t,w)
\in L^\infty(\br\times T^*\br^n)$ with the property that $w \mapsto
 {\varrho}_T(t,w)$ is uniformly Lipschitz continuous, and
\begin{align}\label{r0prop0}
&|{\varrho}_T| \le m\\
&T\partial_t({\delta}_0 + {\varrho}_T) \ge m/2 \qquad \text{in
  $\Cal D'(\br)$}
 \label{r0prop1}
\end{align}
when $|t| <T$. 
\end{prop}

\begin{proof}
(We owe this argument to Lars H\"ormander ~\cite{ho:NT}.) Let
\begin{equation}\label{r0def}
 {\varrho}_{T}(t,w) = \sup_{-T \le s
 \le t}\left({\delta}_0(s,w) - {\delta}_0(t,w) + \frac{1}{2T} \int_s^t
 m(r,w)\,dr - m(s,w)\right)
\end{equation}
for $|t| \le T$, then
\begin{multline*}
{\delta}_0(t,w) + {\varrho}_T(t,w) = \sup_{-T \le s
 \le t}\left({\delta}_0(s,w) - \frac{1}{2T}\int_0^s m(r,w)\,dr -
 m(s,w)\right) \\+ 
\frac{1}{2T} \int_0^t m(r,w)\,dr
\end{multline*}
which immediately gives ~\eqref{r0prop1} since the
supremum is non-decreasing. Since  $w \mapsto {\delta}_0(t,w)$ and $w \mapsto
m(t,w)$ are uniformly   
Lipschitz continuous by Proposition~\ref{h0slow}, we
find by taking the supremum that $w \mapsto {\varrho}_T(t,w)$
is uniformly Lipschitz continuous.
We find from Proposition~\ref{qmaxpropo} that
\begin{equation*}
{\delta}_0(s,w) - {\delta}_0(t,w) + \frac{1}{2T}\int_s^t m(r,w)\,dr -
m(s,w) \le m(t,w) \qquad -T \le s \le t \le T
\end{equation*}
By taking the supremum, we obtain that $ -m(t,w) \le
{\varrho}_T(t,w) \le m(t,w)$ when $|t| \le T$, which proves the result.
\end{proof}

\section{The Wick quantization}\label{norm}

In order to define the multiplier we shall use the Wick quantization,
and we shall also define the function spaces that we shall use.  As
before, we shall assume that $g^\sh = 
(g^\sh)^{\sigma}$ and the coordinates are chosen so that $g^\sh(w) =
|w|^2$.  For $a \in L^\infty(T^*\br^n)$ we define the Wick
quantization:
\begin{equation}\label{defwick}
a^{Wick}(x,D_x)u(x) = \int_{T^*\br^n}a(y,{\eta})
{\Sigma}^w_{y,{\eta}}(x,D_x)u(x)\,dyd{\eta}\qquad u \in  \Cal S(\br^n)
\end{equation}
using the projections ${\Sigma}^w_{y,{\eta}}(x,D_x)$ with Weyl symbol
$${\Sigma}_{y,{\eta}}(x,{\xi}) =
{\pi}^{-n}\exp(-g^\sharp(x-y,{\xi}-{\eta}))
$$
(see~\cite[Appendix~B]{de:suff} or~\cite[Section~4]{ln:coh}).
We find that 
$a^{Wick}$: $ \Cal S(\br^n) \mapsto \Cal S'(\br^n)$ so that 
$(a^{Wick})^* = (\overline a)^{Wick}$,
\begin{equation} \label{poswick}
a \ge 0 
\implies
\sw{a^{Wick}(x,D_x)u,u} \ge
0 \qquad u \in \Cal S(\br^n)
\end{equation}
and
$ 
\mn{a^{Wick}(x,D_x)}_{\Cal L(L^2(\br^n))} \le
\mn{a}_{L^\infty(T^*\br^n)}, 
$ 
which is the main advantage with the Wick quantization
(see \cite[Proposition~4.2]{ln:coh}).  
Now if $a_t(x,{\xi}) \in L^\infty(\br \times T^*\br^n)$ depends on a
parameter ~$t$, then we find that 
\begin{equation}\label{intwick}
 \int_\br \sw{a^{Wick}_tu ,u} {\phi}(t)\, dt =
 \sw{A_{\phi}^{Wick}u, u}  \qquad u \in \Cal S(\br^n)
\end{equation}
where $A_{\phi}(x,{\xi}) = \int_\br a_t(x,{\xi}) {\phi}(t)\, dt$. In
fact, if $a \in L^1$ then this follows from the Fubini theorem, in
general we obtain this by cutting off $a_t$ on large sets in $T^*\br^n$ and using
dominated convergence.
We obtain from the definition that
$a^{Wick} = a_0^{w}$ where
\begin{equation}\label{gausreg}
a_0(w) = {\pi}^{-n}\int_{T^*\br^n} a(z)\exp(-|w-z|^2)\,dz 
\end{equation}
is the Gaussian regularization, thus Wick operators with symmetric symbols
have symmetric Weyl symbols.

We also have the following result about the composition of Wick
operators according to the proofs of Proposition~3.4 in ~\cite{ln:wick}
and Lemma~A.1.5 in ~\cite{ln:cutloss}.

\begin{rem}\label{wickcomp}
Let $a(w)$, $b(w) \in L^\infty$, and let $m_1 $, $m_2 $ be bounded weights
for $g^\sh$. If $|a| \le m_1$ and
$|b'| = |\partial b| \le m_2$, then 
\begin{equation}\label{wickcomp1}
 a^{Wick}b^{Wick} = (ab )^{Wick} + r^w
\end{equation}
with $r \in S(m_1m_2,g^\sh)$. In the case when $a$, $b$ are real valued, $|a| \le
m_1$ and $|b''| \le m_2$, we obtain that 
\begin{equation}\label{wickcomp2}
 \re\left(a^{Wick}b^{Wick} \right) = \left(ab  - \frac{1}{2}a'\cdot
 b'\right)^{Wick} + r^w  
\end{equation}
with $r \in S(m_1m_2,g^\sh)$. Here $a'$ is the distributional
derivative of ~$a \in L^\infty$ and $b'$ is Lipschitz continuous, so
the product is well-defined in $L^\infty$. 
\end{rem}  


If $A \in L^\infty(T^*\br^n)$ is an $M \times N$ system, then we can define $A^{Wick}$ 
by ~\eqref{defwick} on $u \in \Cal S(\br^n, \bc^N)$. These
operators have the same properties as the scalar operators, but 
of course we need that $M = N$ in order for ~\eqref{poswick} to hold.

In the following, we shall assume that
$G =Hg^\sh \le g^\sh$ is a slowly varying metric satisfying 
\begin{equation}\label{gentemp}
H(w) \le C_0 H(w_0) (1 + |w-w_0|)^{N_0} 
\end{equation}
and that $m$ is a weight for $G$ satisfying~~\eqref{gentemp} with ~~$H$
replaced by ~~$m$. This means that ~ $G$ and~$m$ are strongly ${\sigma}$
~temperate in the sense of \cite[Definition~7.1]{bc:sob}. Recall the
symbol class $S^+(1,g^\sh)$ defined by Definition~\ref{s+def}.

\begin{prop}\label{propwick} 
Assume that $a \in L^\infty(T^*\br^n)$ is a $N \times N$ system such that $|a| \le Cm$, then
$ a^{Wick} = a_0^w$ where $a_0\in S(m,g^\sh)$ is given by
~\eqref{gausreg}. If $a \in S(m,G)$ for $G = Hg^\sh$ then $a_0 = a$ modulo symbols in
$S(mH,G)$. If $|a| \le Cm$ and $a = 0$ in a fixed
$G$~ball with center ~$w$, then $a \in S(mH^N,G)$ near ~$w$ for any $N$. If 
$a$ is Lipschitz continuous then we have $a_0\in S^+(1, g^\sh)$.
If $a(t,w)$ and $g(t,w) \in L^\infty(\br\times T^*\br^n)$ are  $N \times N$ systems
and $\ddt a(t,w) \ge g(t,w)$ in $\Cal D'(\br)$ for almost all $w
\in T^*\br^n$, then we find $\sw{\ddt (a^{Wick})u,u} \ge \sw{g^{Wick}u,u}$
in $\Cal D'(\br)$ for $u \in \Cal S(\br^n, \bc^N)$.
\end{prop} 

Observe that the
results are uniform in the metrics and weights.
By localization we find, for example, that if $|a| \le Cm$ and $a \in
S(m,G)$ in a $G$~neighborhood of ~$w_0$, then $a_0 =a$ modulo
$S(mH,G)$ in a smaller $G$~neighborhood of ~$w_0$. These results are
well known, but for convenience we give a short proof. 

\begin{proof} 
Since $a$ is measurable satisfying $|a| \le C m$, where $m(z) \le C_0
m(w)(1 + |z-w|)^{N_0}$ by ~\eqref{gentemp}, we find that
$a^{Wick} = a_0^{w}$ where $a_0 = \Cal O(m)$ is given by
~\eqref{gausreg}. By differentiating on the exponential factor, we find
$a_0 \in S(m,g^\sh)$.

If $a = 0$ in a $G$
ball of radius ${\varepsilon}>0$ and center at $w$, then we can write
\begin{equation*}
{\pi}^{n}a_0(w) = 
\int_{|z-w| \ge {\varepsilon}H^{-1/2}(w)} a(z)\exp(-|w-z|^2)\,dz =\Cal
O(m(w) H^N(w)) 
\end{equation*}
for any $N$ even after repeated differentiation. If $a \in S(m,G)$ then
Taylor's formula gives
\begin{equation*}
a_0(w) = a(w) + {\pi}^{-n} \int_0^1\int_{T^*\br^n} (1-{\theta})\w{a''(w +
{\theta}z)z,z} e^{-|z|^2}\,dzd{\theta}
\end{equation*}
where $a'' \in S(mH,G)$ because $G = Hg^\sh$.  Since $m(w + {\theta}z)
\le C_0 m(w)(1 + |z|)^{N_0}$ and $H(w + {\theta}z) \le C_0H(w)(1 +
|z|)^{N_0} $ when $|{\theta}| \le 1$, we find that $a_0(w) = a(w)$
modulo $S(mH,G)$.  
Now, the  Lipschitz continuity of $a$ means that 
$\partial a \in L^\infty(T^*\br^n)$. 
Since $\partial a_0(w) = {\pi}^{-n}\int_{T^*\br^n}
\partial a(z)\exp(-|w-z|^2)\,dz $, we obtain that $a_0\in S^+(1, g^\sh)$.

For the final claim, we note that $ -\int a(t,w) {\phi}'(t)\,dt \ge \int
g(t,w){\phi}(t)\,dt$ for all $0 \le {\phi} \in C_0^\infty(\br) $ and
almost all $w\in T^*\br^n$, which by~\eqref{poswick}
and~\eqref{intwick} gives
\begin{equation*}
-\int \sw{a^{Wick}(t,x,D_x)u,u} {\phi}'(t)\,dt \ge \int
\sw{g^{Wick}(t,x,D_x)u,u}{\phi}(t)\,dt \qquad 0 \le {\phi} \in
  C_0^\infty(\br)
\end{equation*}
for $u \in  \Cal S(\br^n, \bc^N)$.
\end{proof}

We shall compute the Weyl symbol for the Wick operator
$({\delta}_0 + {\varrho}_T)^{Wick}$, where ~${\varrho}_T$ is given by
Proposition~\ref{apsdef}.  In the
following we shall suppress the $t$~~ variable.

\begin{prop}\label{wickweyl}
Let $B = {\delta}_0 + {\varrho}_0$, where ${\delta}_0$ is given by
Definition~~\ref{d0deforig} and ${\varrho}_0$ is 
real valued and Lipschitz continuous,
satisfying $|{\varrho}_0| \le m$, where ~$m \le
H^{1/2}\w{{\delta}_0}^2/2 \le \w{{\delta}_0}/2$ is a weight
for~$g^\sh$. Then we find 
\begin{equation*}
B^{Wick} = b^w 
\end{equation*}
where $b= {\delta}_1 + {\varrho}_1 \in S(\w{{\delta}_0},g^\sh)\bigcap 
S^+(1,g^\sh)$ is real, ${\delta}_1 \in S(H^{-1/2},g^\sh)\bigcap 
S^+(1,g^\sh)$, and ${\varrho}_1 \in S(m, g^\sh)\bigcap S^+(1, g^\sh)$.
Also, there exists ${\kappa}_2 >0$ so that ${\delta}_1 = {\delta}_0$
modulo $S(H^{1/2},G)$ when $\w{{\delta}_0} \le
{\kappa}_2H^{-1/2}$, which gives $b = {\delta}_0$ modulo
$S(H^{1/2}\w{{\delta}_0}^2, g^\sh)$.
For any ${\lambda} > 0$ we find that $|{\delta}_0| \ge
{\lambda} H^{-1/2}$ and $H^{1/2} \le {\lambda}/3$ imply
that $\sgn(B) = \sgn ({\delta}_0)$ and $ 
|B| \ge {\lambda}H^{-1/2}/3$. 
\end{prop}

\begin{proof}
Let ${\delta}_0^{Wick} = {\delta}_1^w$ and ${\varrho}_0^{Wick} =
{\varrho}_1^w$.  Since $|{\delta}_0|\le \w{{\delta}_0} \le H^{-1/2}$, $|{\varrho}_0|
\le m \le \w{{\delta}_0}/2$ and the symbols are real valued, we obtain from
Proposition~\ref{propwick} that $b \in
S(\w{{\delta}_0},g^\sh)$, ${\delta}_1 \in S(H^{-1/2},g^\sh)$
and ${\varrho}_1 \in S(m, g^\sh)$ are real valued. Since
${\delta}_0$ and ${\varrho}_0$ are 
uniformly Lipschitz continuous, we find that ${\delta}_1$ and ${\varrho}_1 \in
S^+(1,g^\sh)$ by Proposition~\ref{propwick}.

If $\w{{\delta}_0} \le {\kappa}H^{-1/2}$ at ~$w_0$ for sufficiently
small ${\kappa} >0$, then we find by the Lipschitz continuity of~
${\delta}_0$ and the slow variation of ~$G$ that $\w{{\delta}_0} \le
C_0{\kappa}H^{-1/2}$ in a fixed $G$ neighborhood
${\omega}_{\kappa}$ of ~$w_0$ (depending on ${\kappa}$). For ${\kappa}
\ll 1$ we find ${\delta}_0 \in S(H^{-1/2}, G)$ in ${\omega}_{\kappa}$ by
Proposition~\ref{ffactprop}, thus ${\delta}_1 =
{\delta}_0$ modulo $S(H^{1/2},G)$ near ~$w_0$ by
Proposition~\ref{propwick} after localization.

When $|{\delta}_0| \ge {\lambda}H^{-1/2}  > 0$ at
$w_0$, then we find that 
$$|{\varrho}_0| \le m \le \w{{\delta}_0}/2 \le
(1 + H^{1/2}/{\lambda})|{\delta}_0|/2$$
We obtain that $|{\varrho}_0| \le 2|{\delta}_0|/3$ so $\sgn(B) = \sgn ({\delta}_0)$ and
$|B| \ge |{\delta}_0|/3 \ge
{\lambda}H^{-1/2}/3 $ 
when $H^{1/2}  \le {\lambda}/3$, which completes the proof.
\end{proof}

Let $m$ be given by Definition~\ref{h0def}, then
$m$ is a weight for $g^\sh$ according to
Proposition~\ref{h0slow}. We are going to use the symbol classes
$S(m^k, g^\sh)$, $k \in \br$.

\begin{defn}\label{Hdef}
Let  $H(m^k,g^\sh)$,
 be the Hilbert space given by
~\cite[Definition~4.1]{bc:sob} so that
\begin{equation}\label{7.7}
u \in  H(m^k,g^\sh) \iff
a^wu \in L^2  \qquad  \forall\, a \in
S(m^k, g^\sh)\qquad k \in \br
\end{equation}
We let $\mn{u}_{{k}}$ be the norm of 
$H(m^{k},g^\sh)$.
\end{defn}

This Hilbert space has the following properties: $\Cal S$ is dense in
$H(m^k,g^\sh)$, the dual of $H(m^k,g^\sh)$ is naturally identified
with $H(m^{-k},g^\sh)$, and if $u \in H(m^k,g^\sh)$ then $u =
a_0^w v$ for some $v \in L^2(\br^n)$ and $a_0 \in S(m^{-k}, g^\sh)$ 
(see~\cite[Corollary~6.7]{bc:sob}). 
It follows that $a^w \in \op
S(m^k, g^\sh)$ is
bounded: 
\begin{equation}\label{nest0}
 u \in H(m^j,g^\sh) \mapsto a^wu \in 
H(m^{j-k},g^\sh) 
\end{equation} 
with bound only depending on the
seminorms of ~$a$.

We recall Proposition~6.5 in ~\cite{de:cut}, which shows that the topology in $H(m^{1/2},
g^\sh)$ can be defined by the operator $m^{Wick}$.

\begin{prop}\label{wprop} 
Let $B = {\delta}_0 + {\varrho}_0$, where
${\delta}_0$ is given by
Definition~~\ref{d0deforig} and $|{\varrho}_0| \le m$.
Then there exist positive constants $c_1$, $c_2$ and $C_0$ such that
\begin{equation}\label{west3}
c_1 h^{1/2}(\mn{B^{Wick}u}^2 + \mn u^2) \le c_2 \mn{u}_{{1/2}}^2 \le
\sw{m^{Wick}u,u}\le 
C_0 \mn{u}_{{1/2}}^2\qquad u \in \Cal S(\br^n)
\end{equation}
The constants only depend on the seminorms of $f$
in $L^\infty(\br, S(h^{-1},hg^\sh))$.
\end{prop}

In the following, we let $\mn {u(t)}$ be the $L^2$ norm of $x \mapsto
u(t,x) \in \bc^N$ in
$\br^{n}$ for fixed ~ $t$, and~ $\sw{u(t),v(t)}$ the corresponding
sesquilinear inner product. Let $\Cal B = \Cal B(L^2(\br^{n}), \bc^N)$ be
the set of bounded operators $L^2(\br^{n}, \bc^N)\mapsto L^2(\br^{n}, \bc^N)$.
We shall use operators which depend measurably on $t$ in the following
sense.

\begin{defn} 
We say that $t \mapsto A(t)$ is weakly measurable if $A(t) \in \Cal B$
for all ~$t$ and $t \mapsto A(t)u$ is weakly measurable for every $u
\in L^2(\br^{n}, \bc^N)$, i.e., $t \mapsto \sw{A(t)u,v}$ is measurable for
any ~$u$, $v \in L^2(\br^{n}, \bc^N)$. We say that $A(t) \in
L_{loc}^\infty(\br, \Cal B)$ if $t \mapsto A(t)$ is weakly measurable and
locally bounded in ~$\Cal B$. 
\end{defn}

If $A(t) \in L_{loc}^\infty(\br, \Cal B)$, then we find that the
function $t \mapsto \sw{A(t)u,v}\in L_{loc}^\infty(\br)$ has weak
derivative $\dtt \sw{A u,v} \in \Cal D'(\br)$ for any $u$, $v \in
L^2(\br^{n}, \bc^N)$, given by
$$\dtt \sw{A u,v}({\phi}) =
-\int\sw{A(t)u,v}{\phi}'(t)\,dt \qquad {\phi}(t) \in C_0^\infty(\br)
$$ 
If $u(t)$, $v(t) \in L^\infty_{loc}(\br,L^2 (\br^{n}, \bc^N))$ and
$A(t) \in L^\infty_{loc}(\br,\Cal B)$, then $t \mapsto
\sw{A(t)u(t),v(t)} \in L^\infty_{loc}(\br)$ is measurable.
We shall use the following multiplier estimate from ~\cite{de:NT}.

\begin{prop} \label{psest}
Let $P = D_t + iF(t)$ with $F(t) \in L^\infty_{loc}(\br, \Cal B)$. Assume that
$B(t)= B^*(t)  \in L_{loc}^\infty(\br, \Cal B)$, such that
\begin{equation}\label{multestcond} 
\dtt \sw{ B u,u}  + 2\re\sw{B u,F u} \ge 
 \sw{M u,u}\quad\text{in $\Cal D'(I)\qquad \forall\ u \in \Cal
   S(\br^{n}, \bc^N)$} 
\end{equation}
where $M(t) = M^*(t) \in L_{loc}^\infty(\br, \Cal B)$ and $I \subseteq
\br$ is open.  Then we have
\begin{equation}\label{pest}
\int  \sw{M u,u} \,dt \le 2\int \im\sw{Pu,B u} \,dt
\end{equation}
for $u \in C_0^1(I, \Cal S(\br^{n}, \bc^N))$.
\end{prop}

\begin{proof}
Since $B(t)\in L_{loc}^\infty(\br, \Cal B)$, we may for $u$, $v \in \Cal
S(\br^{n}, \bc^N)$ define the regularization 
\begin{equation*}
\sw{B_{\varepsilon}(t)u,v} = {\varepsilon}^{-1} \int
\sw{B(s)u,v}{\phi}((t-s)/{\varepsilon})  \,ds
= \sw{Bu,v}({\phi}_{{\varepsilon},t})\qquad{\varepsilon}>0
\end{equation*} 
where ${\phi}_{{\varepsilon},t}(s) = {\varepsilon}^{-1}
{\phi}((t-s)/{\varepsilon})$ with $0 \le {\phi} \in C_0^{\infty}(\br)$
satisfying $\int {\phi}(t)\,dt = 1$. Then $t \mapsto
\sw{B_{\varepsilon}(t)u,v}$ is in $C^\infty(\br)$ with derivative
equal to $\dtt \sw{Bu,v}({\phi}_{{\varepsilon},t}) =
-\sw{Bu,v}({\phi}_{{\varepsilon},t}')$.  Let $I_0$ be an
open interval such that $I_0 \Subset I$.  Then for small enough
${\varepsilon}>0$ and  $t \in I_0$ we find from condition~\eqref{multestcond} 
that
\begin{equation}\label{newestcond}
  \dtt \sw{ B_{\varepsilon}(t)u,u} +
 2\re\sw{B u,F u}({\phi}_{{\varepsilon},t}) \ge 
 \sw{M u,u}({\phi}_{{\varepsilon},t})\qquad u \in
 \Cal S(\br^{n}, \bc^N)
\end{equation}
In fact, ${\phi}_{{\varepsilon},t} \ge 0$ and $\supp
{\phi}_{{\varepsilon},t} \in C_0^\infty(I)$ for small enough
${\varepsilon}$ when $t \in I_0$.

Now for $u(t)\in C_0^1(I_0,\Cal S(\br^{n}, \bc^N))$ and 
${\varepsilon}>0$ we define
\begin{equation}\label{eq:def1}
B_{\varepsilon,u}(t) = \sw{B_{\varepsilon}(t)u(t),u(t)} =
{\varepsilon}^{-1} \int
\sw{B(s)u(t),u(t)}{\phi}((t-s)/{\varepsilon}) \,ds
\end{equation}
For small enough ${\varepsilon}$ we obtain $B_{\varepsilon,u}(t) \in
C_0^1(I_0)$, with derivative 
 $$\dtt B_{\varepsilon,u} = \sw{(\dtt
B_{\varepsilon})u,u} + 2\re\sw{B_{\varepsilon} u,\ddt u}
 $$ 
since $B(t) \in L_{loc}^\infty(\br, \Cal B)$. By integrating with
respect to $t$, we obtain the vanishing average
\begin{equation}\label{aveq}
0 =  \int \dtt B_{\varepsilon,u}(t)\,dt  =  \int  \sw{(\dtt
  B_{\varepsilon})u,u}\,dt  + \int 2\re\sw{B_{\varepsilon} u,\ddt u}\,dt
\end{equation}
when $u\in C_0^1(I_0,\Cal S(\br^{n}, \bc^N))$. We
obtain from ~\eqref{newestcond} and~\eqref{aveq} that
\begin{equation*}
0\ge \iint \big(\sw{M(s) u(t),u(t)} + 2\re\sw{B(s)u(t), \ddt u(t) -F(s)u(t)}\big)
{\phi}((t-s)/{\varepsilon})\,dsdt
\end{equation*}
By letting ${\varepsilon} \to 0$, we find by dominated convergence that
\begin{equation*}
0 \ge \int \sw{M(t)u(t),u(t)} +2\re\sw{B(t) u(t),\ddt u(t)- F(t)u(t)} \,dt
\end{equation*}
since $u\in C_0^1(I_0,\Cal S(\br^{n}, \bc^N))$
and $M(t)$, $B(t)$, $F(t) \in L^\infty_{loc}(\br, \Cal B)$. Here $\ddt u -Fu =
iPu$ and $ 2\re\sw{Bu,iPu} = 
-2\im\sw{Pu,Bu}$, thus we obtain~\eqref{pest}
for $u\in C_0^1(I_0,\Cal S(\br^{n}, \bc^N))$. Since $I_0$ is an arbitrary open
subinterval with compact closure in ~$I$, this completes the proof of
the proposition.
\end{proof}

\section{The lower bounds}\label{lower}

In this section shall prove Proposition~\ref{mainprop}, which means
obtaining lower bounds on 
\begin{equation*}
 2\im\sw{P_0u,b_T^wu} = \sw{\partial_t b_T^wu,u} + 2\re\sw{F^wu,b_T^wu}
\end{equation*}
where $P_0 = D_t\id_N + iF^w(t,x,D_x)$ with
\begin{equation}\label{sysymbdef}
 F(t,w) = f(t,w)\id_N + F_0(t,w)
\end{equation}
Here $f \in L^\infty(\br, S(h^{-1}, hg^\sh))$ is real valued
satisfying condition $(\ol {\Psi})$ given by~(2.2), $F_0 \in
C^\infty(\br, S(1, hg^\sh))$ and $b_{T}^w = B_{T}^{Wick}$ is the
symmetric scalar operator given by Proposition~\ref{wickweyl} for this
~$f$. Since Proposition~\ref{apsdef} and Proposition~\ref{propwick} give lower bounds on the first
term:
\begin{equation*}
 \partial_t b_T^w = \partial_t B_T^{wick} \ge m^{wick}/2T \qquad \text{in
   $L^2$} \quad |t| \le T
\end{equation*}
it only remains to obtain comparable lower bounds on $\re b_T^wF^w$ by
Proposition~\ref{wprop}.

By Claim~\ref{subrem} we may also assume that 
\begin{equation}\label{subsymbdef}
 F_0 = \w{d_wf, R} = \sum_j \partial_{w_j}f R_{j}\qquad\text{modulo
   $S(h, hg^\sh)$} \qquad \forall\,t
\end{equation}
where $R_j \in S(h^{1/2}, hg^\sh)$ are $N\times N$ systems, $\forall\,
j$. Observe that since $d_wf \in S(MH^{1/2}, G)$, $hg^\sh \le 3G$ and $h \le
MH^{1/2}h^{1/2}$ by~\eqref{Mdef0} we find that $F_0 \in
S(MH^{1/2}h^{1/2}, G) \subseteq S(1,G)$ and thus $F \in S(M, G)$.

In the following, the results will hold
for almost all $|t| \le T$ and will
only depend on the seminorms of $f$ in $L^\infty(\br,
S(h^{-1},hg^\sh))$. We shall suppress the $t$ variable and assume the
coordinates chosen so that $g^\sh(w) = |w|^2$.
In order to prove Proposition~\ref{mainprop} we need to prove the
following result.

\begin{prop}\label{lowersign}
Assume that $F$ is given by~\eqref{sysymbdef}--\eqref{subsymbdef} and
$B = {\delta}_0 + {\varrho}_0$. Here ${\delta}_0$ is given by 
Definition~~\ref{d0deforig}, ${\varrho}_0$ is 
real valued and Lipschitz continuous
satisfying $|{\varrho}_0| \le m$, where $m \le \w{{\delta}_0}/2$ is given by
Definition~\ref{h0def}. Then we have
\begin{equation}\label{lowerbound}
\re\sw{ B^{Wick} F^wu,u} \ge \sw{C^w u,u} \qquad\forall\  u \in
\Cal S(\br^n, \bc^N)
\end{equation}
for some $N \times N$ system $C \in S(m, g^\sh)$.
\end{prop}

\begin{proof}[Proof of Proposition~\ref{mainprop}]
Let $B_T = {\delta}_0 + {\varrho}_T$, where ${\delta}_0 +
{\varrho}_T$ is the pseudo-sign for~ $f$ given by
Proposition~\ref{apsdef} for $0 < T \le 1$, so that $|{\varrho}_{T}|
\le m$ and
\begin{equation}\label{dtbt} 
{\partial_t } ({\delta}_0 + {\varrho}_{T}) \ge 
m/2T\qquad\text{in $\mathcal D'\big(]{-T},T[\big)$}
 \end{equation}
If we put $B_T \equiv 0$ when $|t| >T$, then
$B_T^{Wick} = b_T^w$ where $b_T(t,w) \in L^\infty(\br,
S(H^{-1/2}, g^\sh)\linebreak[3] \bigcap S^+(1,
g^\sh)) 
$ uniformly by Proposition~\ref{wickweyl}.
We find by Proposition~\ref{propwick}
and ~\eqref{dtbt} that
\begin{equation}\label{dbest}
 \sw{({\partial_t}B_T)^{Wick} u,u} \ge
 \sw{m^{Wick}u,u}/2T \qquad\text{in  $\mathcal D'\big(]{-T},T[\big)$}
\end{equation}
when $u \in  \Cal S({\mathbf R}^n)$.  We obtain from
Proposition~\ref{wprop} that there exist positive constants $c_1$ and
$c_2$ so that
\begin{equation}\label{mlow}
\sw{m^{Wick}u,u} \ge c_2\mn{u}_{{1/2}}^2 \ge
c_1h^{1/2} (\mn{b_T^wu}^2 + \mn{u}^2)\qquad 
u \in  \Cal S({\mathbf R}^n)
\end{equation}
Here $\mn{u}_{{1/2}}$ is the norm of the Hilbert space
$H(m^{1/2}, g^\sh)$ given by Definition~\ref{Hdef}. By
Proposition~\ref{lowersign}, we find for 
almost all $t \in [-T,T]$ that
\begin{equation}
\re \sw{(B_T^{Wick}F^w)\restr{t}u,u} = 
\sw{C^w(t)u,u} \quad u \in  \Cal S({\mathbf R}^n, \bc^N)
\end{equation}
here the $N \times N$ system $C(t) \in S(m, g^\sh)$ uniformly. We obtain
from~\eqref{nest0}, \eqref{mlow} and duality that there exists a
positive constant $c_3$ such that
\begin{equation}\label{cest}
|\sw{C^w(t)u,u}| \le \mn{u}_{{1/2}} \mn{C^w(t)
    u}_{-1/2} \le c_3
\mn{u}_{{1/2}}^2 \le c_3\sw{m^{Wick}u,u}/c_2
\end{equation}
for $u \in  \Cal S({\mathbf R}^n, \bc^N)$ and $|t| \le T$. 
We find from ~\eqref{dbest}--\eqref{cest} that
\begin{equation*}
\sw{\partial_{t}b^w_Tu,u} +
2\re\sw{F^w u,b^w_{T}u} \ge (1/2T -
2c_3/c_2)\sw{m^{Wick}u,u} \qquad\text{ in $\Cal D'\big(]{-T},T[\big)$}
\end{equation*}
for $u \in \Cal S(\br^{n}, \bc^N)$. 
By using Proposition~\ref{psest} with $P = D_t\id_N + iF^w(t,x,D_x)$, $B
= b_T^w$ and $M = m^{Wick}/4T$ we obtain that
\begin{equation*}
c_1h^{1/2}\int \mn{b_t^wu}^2 + \mn u^2 \,dt \le \int \sw{m^{Wick}u,u} \,dt \le
 {8T} \int \im\sw{P_0u,b^w_Tu} \,dt
\end{equation*}
if $u \in \Cal S(\br\times \br^{n}, \bc^N)$ has support where $|t| < T \le
c_2/8c_3$. This finishes the proof of Proposition~\ref{mainprop}.
\end{proof}

\begin{proof}[Proof of Proposition~\ref{lowersign}]
First we note that since $B^{Wick} = b^w \in \op
S(\w{{\delta}_0}, g^\sh)$ by Proposition~\ref{wickweyl} and
$h^{1/2}\w{{\delta}_0}^2 \le 6m$ by~\eqref{hhhest}, 
we find $B^{Wick}R^w \in \op S(m,g^\sh)$ when $R \in
S(h^{1/2},g^\sh)$. Since $\im F = \frac{1}{2i}(F- F^*) \in S(1,hg^\sh)$ we find
$$
2\re (B^{Wick} i(\im F)^w) = i[b^w, (\im F)^w ] \in \op
S(h^{1/2},g^\sh)
$$
thus it suffices to consider symmetric $F$ satisfying~\eqref{subsymbdef}.

We shall localize in $T^*\br^n$ with respect to the metric $G = Hg^\sh$,
and estimate the localized operators.
We shall use the neighborhoods 
\begin{equation}\label{omegajdef} 
 {\omega}_{w_0}({\varepsilon}) = \set{w:\ |w-w_0| <
{\varepsilon}H^{-1/2}(w_0)}\qquad  \text{for } w_0 \in T^*\br^n
 \end{equation}
We may in the following assume that 
${\varepsilon}$ is small enough so that 
$w \mapsto H(w)$ and $w \mapsto M(w)$ only vary with a fixed factor
in~ ${\omega}_{w_0}({\varepsilon})$.  Then
by the uniform Lipschitz continuity of
$w \mapsto {\delta}_0(w)$ we can find ${\kappa}_0 >0$
with the following property: for $0 < {\kappa}\le {\kappa}_0$ there
exist positive constants $c_{\kappa}$ and ${\varepsilon}_{\kappa}$ so
that for any $w_0 \in T^*\br^n$ we have
\begin{alignat}{2}
&|{\delta}_0(w)| \le {\kappa}H^{-1/2}(w)&\qquad &w \in
{\omega}_{w_0}({\varepsilon}_{\kappa})  \qquad\text{or}\label{d0case}\\
&|{\delta}_0(w)| \ge c_{\kappa} H^{-1/2}(w)&\qquad &w \in
{\omega}_{w_0}({\varepsilon}_{\kappa}) \label{d1case}
\end{alignat}
In fact, we have by the Lipschitz continuity that $|{\delta}_0(w) -
{\delta}_0(w_0)| \le {\varepsilon_{\kappa}}H^{-1/2}(w_0)$ when
$w\in {\omega}_{w_0}({\varepsilon}_{\kappa})$. Thus, if
${\varepsilon}_{\kappa} \ll {\kappa}$ we obtain that ~\eqref{d0case}
holds when $|{\delta}_0(w_0)| \ll {\kappa}H^{-1/2}(w_0)$ and
~\eqref{d1case} holds when $|{\delta}_0(w_0)| \ge
c{\kappa}H^{-1/2}(w_0)$.

Let ${\kappa}_1$ be given by Proposition~\ref{ffactprop},
${\kappa}_2$ by Proposition~\ref{wickweyl}, and
let ${\varepsilon}_{\kappa}$ and $c_{\kappa}$ be given by
\eqref{d0case}--\eqref{d1case} for ${\kappa} = \min
({\kappa}_0, {\kappa}_1, {\kappa}_2)/2$. 
Using Proposition~\ref{wickweyl} with ${\lambda} = c_{\kappa}$
we obtain that $\sgn(B) = \sgn ({\delta}_0)$ and 
\begin{equation} \label{kappa3ref}
|B |\ge c_{\kappa}H^{-1/2}/3\qquad\text{in ~$
{\omega}_{w_0}({\varepsilon}_{\kappa})$} 
\end{equation}
if $H^{1/2} \le c_{\kappa}/3$ and ~\eqref{d1case} holds in ~$
{\omega}_{w_0}({\varepsilon}_{\kappa})$. 

Choose real symbols
$\set{{\psi}_j(w)}_j$ and $\set{{\Psi}_j(w)}_j
\in S(1,G)$ with values in ~~$\ell^2$, such that 
$\sum_{k} {\psi}_j^2 \equiv 1$, ${\psi}_j{\Psi}_j = {\psi}_j$, 
${\Psi}_j = {\phi}_j^2 \ge 0$ for some
$\set{{\phi}_j(w)}_j \in S(1,G)$ with values in ~~$\ell^2$ so that
$$\supp {\phi}_j \subseteq
{\omega}_j = {\omega}_{w_j}({\varepsilon}_{\kappa})
$$ 
Recall that $B^{Wick} = b^w$ where $b= {\delta}_1 + {\varrho}_1$ is
given by Proposition~\ref{wickweyl}. In particular, ${\delta}_1 \in
S(H^{-1/2}, G)$ when $H^{1/2} \le {\kappa}_2/2$ and~\eqref{d0case}
holds, since then $\w{{\delta}_0} \le {\kappa}_2 H^{-1/2}$.

\begin{lem}\label{bfconglem}
We find that $A_j = {\Psi}_jb \re F \in S(MH^{-1/2},g^\sharp)\bigcap
S^+(M,g^\sharp)$ uniformly in $j$, and
 \begin{equation}\label{bfcong}
\re (b^wF^w) =  \sum_{j}^{} {\psi}_j^w A_j^w{\psi}_j^w  
\qquad \text{modulo $\op S(m,g^\sharp)$}
\end{equation}
We have $A_j^w = \re b^wF_j^w$ modulo $\op
S(m,g^\sharp)$ uniformly in ~$j$, where $F_j = {\Psi}_jF$.
\end{lem}

\begin{proof}
Since $b \in S(H^{-1/2}, g^\sh) \bigcap S^+(1, g^\sh)$, ${\psi}_j
\in S(1, G)$ and $F_j \in S(M, G)$ we obtain that 
$A_j  \in S(MH^{-1/2},g^\sharp)\bigcap
S^+(M,g^\sharp)$ uniformly in $j$.
Proposition~\ref{mestprop}
gives that
\begin{equation} \label{fprimest}
MH^{3/2}\w{{\delta}_0}^2 \le Cm
\end{equation} 
thus we may ignore terms in $\op
S(MH^{3/2}\w{{\delta}_0}^2,g^\sharp)$. Observe that since $b \in 
S(H^{-1/2},g^\sh)$, $\set{{\psi}_k}_k \in S(1,G)$ has values in
$\ell^2$ and $A_k \in S(MH^{-1/2},g^\sh)$ uniformly, 
Lemma~\ref{calcrem} and Remark~\ref{vvcalc} gives that the symbols of
$b^wF^w$, $b^wF_j^w$ and $\sum_k {\psi}_k^wA_k^w{\psi}_k^w$ have expansions in
$S(MH^{j/2}, g^\sh)$. Also observe that in the domains ~${\omega}_j$
where $H^{1/2} \ge c > 0$, we find from Remark~\ref{vvcalc} that the
symbols of $\sum_{k} {\psi}_k^w A_k^w{\psi}_k^w$, $b^w F_j^w$ and
$b^wF^w$ are in $ S(MH^{3/2}, g^\sharp)$ giving the result in this
case. Thus, in the following, we shall assume that $H^{1/2}\ll 1$,
and we shall consider the neighborhoods where~\eqref{d0case}
or~\eqref{d1case} holds.

If~\eqref{d1case} holds then we find that $\w{{\delta}_0} \cong
H^{-1/2}$ so $S(MH^{1/2},g^\sharp) \subseteq S(m, g^\sharp)$
in ${\omega}_j$ by ~\eqref{fprimest}.  Since $b\in S^+(1,g^\sharp)$
and $A_j \in S^+(M,g^\sh)$ we find from Lemma~\ref{calcrem} and
Remark~\ref{vvcalc} that the symbols of both $\re b^w F^w $ and
$\sum_{k}^{} {\psi}_k^w A_k^w {\psi}_k^w$ are equal to $\sum_k
{\psi}_k^2 A_k = \re bF$ modulo $S(MH^{1/2}, g^\sharp)$ in
~${\omega}_j$. We also find that the symbol of $\re b^w  F_j^w$ is equal to~
$A_j$ modulo $S(MH^{1/2},g^\sharp)$, which proves the result in this
case.

Next, we consider the case when ~\eqref{d0case} holds with ${\kappa} =
\min ({\kappa}_0, {\kappa}_1, {\kappa}_2)/2$ and $H^{1/2} \le
{\kappa}_2/2$ in~${\omega}_j$. Then $\w{{\delta}_0} \le {\kappa}_2
H^{-1/2}$ so $b = {\delta}_1 + {\varrho}_1 \in S(H^{-1/2},G)+
S(m, g^\sharp)$ in~${\omega}_j$ by Proposition ~\ref{wickweyl}.
Now $b$ is real and $F$ is symmetric modulo $S(MH,G)$. Thus,
by taking the symmetric part of $b^w F^w = {\delta}_1^w F^w +
{\varrho}_1^w F^w$ we obtain from Lemma~\ref{calcrem} that the symbol
of $\re (b^w F^w -(bF)^w)$ is in $S(MH^{3/2}, G) + S(MHm,
g^\sh) \subseteq S(m,g^\sh)$ in ~${\omega}_j$ since $M \le
CH^{-1}$. Similarly, we find that $A_j^w = \re b^w F_j^w$ modulo
$S(m,g^\sh)$.  Since $A_j \in S(MH^{-1/2},G)+
S(Mm,g^\sharp)$ uniformly, we find that the symbol of $\sum_{k}
{\psi}_k^w A_k^w {\psi}_k^w$ is equal to $\re bF$ modulo $ S(m,g^\sh)$
in ~${\omega}_j$ by Remark~\ref{vvcalc}, which proves ~\eqref{bfcong} and
Lemma~\ref{bfconglem}.
\end{proof}

Next, we shall show that there
exists\/ $N \times N$ system $C_{j} \in S(m, g^\sharp)$ uniformly, such that
\begin{equation}\label{lowerest}
\sw{A_j^w u,u} \ge \sw{C^w_{j} u,u}\qquad
u \in  \Cal S({\mathbf R}^n, \bc^N)
\end{equation} 
Then we obtain from~\eqref{bfcong} and~\eqref{lowerest} that 
$$ 
\re\sw{b^w F^w u,u} \ge \sum_{j} \sw{{\psi}_j^wC^w_{j}{\psi}_j^w u,u}
+\sw{R^wu,u}\qquad\text{$u \in \Cal S({\mathbf R}^n, \bc^N)$}
$$ 
where $\sum_{j}{\psi}_j^wC^w_{j}{\psi}_j^w$ and $R^w \in \op
S(m,g^\sharp)$, which will prove Proposition~\ref{lowersign}.

Thus, it remains to show that there exists $C_{j} \in S(m,
g^\sharp)$ satisfying~\eqref{lowerest}.
As before we are going to consider the
cases when $H^{1/2} \cong 1$ or $H^{1/2} \ll 1$, and when
~\eqref{d0case} or ~\eqref{d1case} holds in ${\omega}_j =
{\omega}_{w_j}({\varepsilon}_{\kappa})$ for ${\kappa} = \min
({\kappa}_0, {\kappa}_1, {\kappa}_2)/2$. When $H^{1/2} \ge
c > 0$ we find that
$A_j \in S(MH^{3/2},g^\sharp)
\subseteq S(m,g^\sharp)$ uniformly by
~\eqref{fprimest} which gives the lemma with $C_j = A_j$ in this
case. Thus, we may assume that 
\begin{equation}\label{kappa4def}
H^{1/2} \le {\kappa}_4 = \min({\kappa}_0, {\kappa}_1, {\kappa}_2,
{\kappa}_3)/2 \qquad \text{in ${\omega}_j$} 
\end{equation}
with ${\kappa}_3  = 2c_{\kappa}/3$
so that \eqref{kappa3ref} follows from ~\eqref{d1case}.

First, we consider the case when $H^{1/2} \le {\kappa}_4$ and
~\eqref{d1case} holds in ${\omega}_j$.  
Since $|{\delta}_0(w)| \ge c_{\kappa} H^{-1/2}(w)$, we find
$\w{{\delta}_0} \cong H^{-1/2}$ in ~${\omega}_j$. As before we may ignore
terms in $S(MH^{1/2},g^\sharp) \subseteq
S(m,g^\sharp)$ in ${\omega}_j$ by
~\eqref{fprimest}. Let $f_j = {\Psi}_jf$, since
$\sgn(f) = \sgn({\delta}_0) = \sgn(B)$ in ${\omega}_j$ by
Proposition~\ref{wickweyl} we find that $f_jB \ge 0$.
Since $f_j \in S(M, G)$, we find $f_j^w = f_j^{Wick}$ modulo
$\op S(MH,G)$ by Proposition~\ref{propwick}, thus we may replace
$f_j^w$ with $f_j^{Wick}$. 
Since $F_{0,j} \in S(MH^{1/2}h^{1/2}, G)$ by ~\eqref{subsymbdef} we find 
that $B^{Wick}F^w_{0,j} \in \op S(MH^{1/2}, g^\sh)$. Since $|B| \le
CH^{-1/2}$ and $B \in S^+(1,g^\sh)$, we find from
~\eqref{wickcomp2}  in Remark ~\ref{wickcomp} that
\begin{equation*}
A_j^w = \re B^{Wick} f_j^{Wick} = (Bf_j)^{Wick} \ge 0
\qquad\text{in $L^2$ modulo $\op S(MH^{1/2},g^\sh)$} 
\end{equation*}
which gives~\eqref{lowerest} in this case.

Finally, we consider the case when ~\eqref{d0case} holds with ${\kappa} =
\min ({\kappa}_0, {\kappa}_1, {\kappa}_2)/2$ and $H^{1/2} \le
{\kappa}_4\le {\kappa}$ in ${\omega}_j$. Then $\w{{\delta}_0} \le
2{\kappa}H^{-1/2}$ so we obtain from Proposition
~\ref{ffactprop} that 
${\delta}_0 \in S(H^{-1/2}, G) \bigcap S(\w{{\delta}_0}, g^\sh)$  in~${\omega}_j$. We have that $b^w =
({\delta}_0 + {\varrho}_0)^{Wick} = B^{Wick}$, where 
\begin{equation} \label{hhhesta}
|{\varrho}_0|
\le m \le H^{1/2}\w{{\delta}_0}^2/2 \le \w{{\delta}_0}/2
\end{equation} 
by Propositions~\ref{h0slow} and ~\ref{apsdef}. Also, we find form Lemma~\ref{bfconglem}
that  $A_j^w = \re B^{Wick} F_j^w$
modulo $\op S(m,g^\sharp)$.

Take ${\chi}(t) \in C^{\infty}(\br)$ such that
$0 \le {\chi}(t) \le 1$, $|t| \ge 2$ in
$\supp {\chi}(t)$ and
${\chi}(t) = 1$ for $|t| \ge 3$.
Let ${\chi}_0 = {\chi}({\delta}_0)$, then 
$2 \le |{\delta}_0|$ and
$\w{{\delta}_0}/|{\delta}_0| \le 3/2$
 in $\supp {\chi}_0$, thus 
\begin{equation}\label{b0ref} 
1 + {\chi}_0{\varrho}_0/{\delta}_0 \ge 1 -
{\chi}_0\w{{\delta}_0}/2|{\delta}_0| \ge 1/4
\end{equation}
Since $|{\delta}_0| \le 3$ in $\supp (1-{\chi}_0)$
we find  by~ \eqref{hhhesta}
that 
$$
B = {\delta}_0 + {\chi}_0{\varrho}_0 = {\delta}_0( 1 + {\chi}_0{\varrho}_0/{\delta}_0)
$$ 
modulo terms that are $\Cal O(H^{1/2})$.
Since $|{\delta}_0'| \le 1$ and
$$|{\chi}_0{\varrho}_0/{\delta}_0| \le {\chi}_0H^{1/2}\w{{\delta}_0}^2/2|{\delta}_0| 
\le 3H^{1/2} \w{{\delta}_0}/4$$ we find from ~\eqref{wickcomp1} in
Remark~\ref{wickcomp} that
\begin{equation}\label{bwref}
 B^{Wick} =  {\delta}_0^{Wick}B_0^{Wick} \qquad\text{modulo $\op
   S(H^{1/2}\w{{\delta}_0}, g^\sh)$}
\end{equation} 
where $B_0 = 1 + {\chi}_0{\varrho}_0/{\delta}_0 = \Cal O(1)$.
Proposition~\ref{wickweyl} gives
$({\chi}_0{\varrho}_0/{\delta}_0)^{Wick} \in \op
S(H^{1/2}\w{{\delta}_0}, g^\sh)$ and
${\delta}_0^{Wick} = {\delta}_1^w $ where  ${\delta}_1 \in
S(H^{-1/2}, g^\sh)$ and ${\delta}_1 = {\delta}_0$
modulo $\op S(H^{1/2}, G)$ in~${\omega}_j$.
Thus Lemma~\ref{calcrem} and ~\eqref{bwref} gives
\begin{equation}\label{nycalcref0}
 B^{Wick} = {\delta}_1^wB_0^{Wick} = {\delta}_0^wB_0^{Wick} + c^w  \qquad\text{modulo $\op
   S(H^{1/2}\w{{\delta}_0}, g^\sh)$}
\end{equation}
where $c \in S(H^{-1/2}, g^\sh)$ such that $\supp c \bigcap
{\omega}_j = \emptyset$. 

We find from Proposition~\ref{ffactprop} that $ f =
{\alpha}_0{\delta}_0 $, where ${\kappa}_1MH^{1/2}\le {\alpha}_0 \in
S(MH^{1/2}, G)$, so Leibniz' rule gives ${\alpha}_0^{1/2} \in S(M^{1/2}H^{1/4},
G)$. Let  $f_j = {\Psi}_jf$ and
\begin{equation}\label{ajdef}
 a_j =
{\alpha}_0^{1/2}{\phi}_j{\delta}_0 \in S(M^{1/2}H^{-1/4},G) \bigcap
S(M^{1/2} H^{1/4}\w{{\delta}_0}, g^\sh)
\end{equation}
Since ${\Psi}_j = {\phi}_j^2$ we find $a_j^2 = f_j{\delta}_0$
and the calculus gives
\begin{equation}\label{nycalcref1}
 a_j^w({\alpha}_0^{1/2}{\phi}_j)^w = f_j^w\qquad\text{modulo $\op S(MH,G)$}
\end{equation}
Since $\supp f_j \bigcap \supp c = \emptyset$ we find that $f_j^wc^w \in \op S(MH^{3/2},g^\sh)$. 
We also have
 \begin{equation}\label{nycalcref2}
\re  f_j^w{\delta}_0^{w} = a_j^w a_j^w\qquad\text{modulo $\op S(MH^{3/2},G)$}
 \end{equation}
and $\im  f_j^w{\delta}_0^{w} \in \op S(MH^{1/2},G)$. We obtain from~
\eqref{nycalcref0} and~\eqref{nycalcref1} that
\begin{equation}\label{W1}
 f_j^wB^{Wick} = f_j^w({\delta}_0^wB_0^{Wick} + c^w + r^w) =
 f_j^w{\delta}_0^wB_0^{Wick} + a_j^wr_j^w \quad\text{modulo $\op S(m,g^\sh)$}
\end{equation}
where $r \in S(H^{1/2}\w{{\delta}_0}, g^\sh)$ which gives $r_j^w =
({\alpha}_0^{1/2}{\phi}_j)^w r^w \in \op
S(M^{1/2}H^{3/4}\w{{\delta}_0}, g^\sh)$. If $\re A = \frac{1}{2}(A +
A^*)$, $\im A = \frac{1}{2i}(A
-A^*)$ and $B^* = B$ then
\begin{equation*}
 \re (AB) = \re (\re A)B + i [\im A, B]/2
\end{equation*}
By taking $A = f_j^w{\delta}_0^w$ and $B = B_0^{Wick}$ we find from
~\eqref{nycalcref2} that  
\begin{equation}\label{W2}
 \re \left(f_j^w{\delta}_0^wB_0^{Wick}\right)  = \re(a_j^wa_j^wB_0^{Wick})
\qquad\text{modulo $\op S(m,g^\sh)$}
\end{equation}
In fact, $B_0 = 1 + {\chi}_0{\varrho}_0/{\delta}_0$ and
$({\chi}_0{\varrho}_0/{\delta}_0)^{Wick} \in \op
S(H^{1/2}\w{{\delta}_0}, g^\sh)$ by Proposition~\ref{propwick}, thus
$$
[a^w,B_0^w] = [a^w,
({\chi}_0{\varrho}_0/{\delta}_0)^{Wick}]
\in \op S(M H^{3/2}\w{{\delta}_0}, g^\sh)
$$
when $a \in S(MH^{1/2}, G)$.
Similarly, we find from ~\eqref{ajdef} that
\begin{equation}\label{W3}
a_j^wa_j^wB_0^{Wick} =
 a_j^w(B_0^{Wick}a_j^w + s_j^w) \qquad\text{modulo $\op S(m,g^\sh)$}
\end{equation}
where $s_j = [a_j^w, B_0^{Wick}] \in S(M^{1/2}H^{3/4}\w{{\delta}_0},
g^\sh)$. Next, we shall use an argument by Lerner ~\cite{ln:cutloss}. 
Since $B_0 \ge 1/4$ by~\eqref{b0ref} we find from 
\eqref{W1}--\eqref{W3} that
\begin{equation}\label{preplow}
 \re f_j^w B^{Wick} \ge \frac{1}{4}a_j^wa_j^w + \re a_j^wS_j^w
 \qquad\text{in $L^2$ modulo $\op S(m,g^\sh)$} 
\end{equation}
where $S_j = r_j + s_j \in 
S(M^{1/2}H^{3/4}\w{{\delta}_0}, g^\sh)$.
Then by completing the square, we find
\begin{equation}\label{compsq}
 \re f_j^w B^{Wick} \ge  \frac{1}{4}\left(a_j^w + 2S_j^w\right)^*
\left(a_j^w + 2S_j^w\right)
 \ge 0 \quad\text{in $L^2$ modulo $\op S(m,g^\sh)$}
\end{equation}
since $(S_j^w)^*S_j^w = \ol S_j^wS_j^w \in
\op S(MH^{3/2}\w{{\delta}_0}^2, g^\sh)$.

But we must also consider $\re F^w_{0,j}B^{Wick}$, where  $F_0$
satisfies ~\eqref{subsymbdef} so
\begin{equation} \label{lowterm}
F_{0,j} =
{\Psi}_jF_0 \in S(MH^{1/2}h^{1/2}, G)
\end{equation}
We shall prove that
\begin{equation}\label{lowerterm}
\re F^w_{0,j}B^{Wick} =  \re a_j^{w}R_j^w\qquad\text{modulo $\op
  S(m,g^\sh)$}
\end{equation}
where $R_j \in S(M^{1/2}H^{3/4}, g^\sh)$, which can then be
included in the term given by $S_j$ in~\eqref{preplow}. Since $b =
{\delta}_0 \in S(H^{-1/2}, G)$
modulo $S(H^{1/2}\w{{\delta}_0}^2, g^\sh)$ in $\omega_j$ by
Proposition~\ref{wickweyl} we find 
\begin{equation*}
 \re F^w_{0,j}B^{Wick} = \re F^w_{0,j}b^w =(\re F_{0,j}{\delta}_0)^w
\end{equation*}
modulo $\op S(m,g^\sh)$. We find from~\eqref{ajdef}
and~\eqref{lowterm} that  $\im a_j = 0$, so 
\begin{equation*}
\re F_{0,j}{\delta}_0 = \re {\phi}_j^2 F_0{\delta}_0 = a_jR_j
\end{equation*}
where 
$$R_j = \re {\phi}_jF_0/{\alpha}_0^{1/2}  \in
S(M^{1/2}H^{1/4}h^{1/2}, G) \subseteq  S(M^{1/2}H^{3/4},G)
$$ 
This gives $(\re F_{0,j}{\delta}_0)^w = a_j^{w}R_j^w$
modulo $\op S(MHh^{1/2}, G) \subseteq \op S(m,g^\sh)$, so we
obtain~\eqref{lowerterm}. 
By adding ~$R_j$ to $S_j$ in~\eqref{preplow} and
completing the square as in ~\eqref{compsq}, we obtain
~\eqref{lowerest} in this case.  
This completes the proof of Proposition~\ref{lowersign}.
\end{proof}

\begin{rem}\label{finalrem}
It follows from the proof of Proposition~\ref{lowersign} 
that in order to obtain the estimate~\eqref{lowerbound} it suffices
that the lower order term  $F_0 \in S(MH, g^\sh) \subseteq S(1,g^\sh)$.
\end{rem}

\bibliographystyle{amsplain}

\begin{thebibliography}{10}

\bibitem{bf}
R.\ Beals and C.\ Fefferman, \emph{On local solvability of linear
  partial differential equations}, Ann. of Math. \textbf{97} (1973), 482--498.

\bibitem{bc:sob}
J.-M.\ Bony and J.-Y.\ Chemin, \emph{Espace fonctionnels associ\'es au
  calcul de {Weyl}-{H\"{o}r\-mander}}, Bull. Soc. Math. France \textbf{122}
  (1994), 77--118.


\bibitem{CPT}
F.\ Colombini, L.\ Pernazza, and F.\ Treves,
  \emph{Solvability and nonsolvability of second-order evolution equations},
  Hyperbolic problems and related topics, Grad. Ser. Anal., Int. Press,
  Somerville, MA, 2003, 111--120.


\bibitem{de:thesis}
N.\ Dencker, \emph{On the propagation of singularities for pseudo-differential
  operators of principal type}, Ark. Mat. \textbf{20} (1982), 23--60.


\bibitem{de:prep}
\bysame, \emph{Preparation theorems for matrix valued functions}, Ann.
  Inst. Fourier (Grenoble) \textbf{43} (1993), 865--892.

\bibitem{de:ln}
\bysame, \emph{The solvability of non ${L}^2$ solvable operators}, Journees
  ''Equations aux D\'eriv\'ees Partielles'', St.\ Jean de Monts, France,
  1996.

\bibitem{de:suff}
\bysame, \emph{A sufficient condition for solvability}, International
  Mathematics Research Notices \textbf{1999:12} (1999), 627--659.


\bibitem{de:NT}
\bysame, \emph{The resolution of the
  Nirenberg-Treves conjecture}, Ann. of Math. \textbf{163} (2006), 405--444.

\bibitem{de:pisa}
\bysame, \emph{The solvability of pseudo-differential operators}, Phase space
analysis of partial differential equations,  Vol. I,
Pubbl. Cent. Ric. Mat. Ennio Giorgi, Scuola Norm. Sup., Pisa, 2004,
 175--200.

\bibitem{de:cut}
\bysame, \emph{On the solvability of  pseudo-differential operators},
S\'emin. Equ. D\'eriv. Partielles, Ecole Polytech., Palaiseau, 2005--2006. 

\bibitem{de:sys}
\bysame, \emph{The pseudospectrum of systems of semiclassical operators}, Anal.
  PDE \textbf{1} (2008), 323--373.

\bibitem{ho:weyl} 
L.\ {H\"{o}rmander}, \emph{The {Weyl} calculus of
pseudo-differential operators}, Comm. Partial Differential Equations
\textbf{32} (1979), 359--443.

\bibitem{ho:suff}
\bysame, \emph{Pseudo-differential operators of principal type}, Nato
Adv.\ Study Inst.\ on Sing.\ in Bound.\ Value Problems, Reidel Publ.\
Co., Dordrecht, 1981, 69--96.

\bibitem{ho:yellow}
\bysame, \emph{The analysis of linear partial differential operators}, vol.
  I--IV, Springer Verlag, Berlin, 1983--1985.

\bibitem{ho:conv}
\bysame, \emph{Notions of convexity}, {Birkh\"{a}user}, Boston, 1994.

\bibitem{ho:solv} \bysame, \emph{On the solvability of
pseudodifferential equations}, Structure of solutions of differential
equations (M.~Morimoto and T.~Kawai, eds.), World Scientific, New
Jersey, 1996, 183--213.

\bibitem{ho:NT}
\bysame, \emph{The proof of the Nirenberg-Treves conjecture according
to N. Dencker och N. Lerner}, Preprint.


\bibitem{ln:2d}
N.\ Lerner, \emph{Sufficiency of condition $({\Psi})$ for local
  solvability in two 
  dimensions}, Ann. of Math. \textbf{128} (1988), 243--258.

\bibitem{ln:ex}
\bysame, \emph{Nonsolvability in ${L}^2$ for a first order operator satisfying
  condition $({\Psi})$}, Ann.\ of Math. \textbf{139} (1994), 363--393.

\bibitem{ln:coh}
\bysame, \emph{Energy methods via coherent states and advanced
  pseudo-differential calculus}, Multidimensional complex analysis and partial
  differential equations (P.~D. Cordaro, H.~Jacobowitz, and S.~Gidikin, eds.),
  Amer. Math. Soc., Providence, R.I., USA, 1997, 177--201.
 
\bibitem{ln:per}
\bysame, \emph{Perturbation and energy estimates}, Ann. Sci. Ecole Norm. Sup.
  \textbf{31} (1998), 843--886.


\bibitem{ln:wick}
\bysame, \emph{The Wick calculus of pseudo-differential operators and some of
  its applications}.  
Cubo Mat. Educ.   \textbf{5}  (2003),  213--236. 


\bibitem{ln:fact}
\bysame, \emph{Factorization and solvability}, Preprint.


\bibitem{ln:cutloss}
\bysame, \emph{Cutting the loss of derivatives for solvability under
  condition (${\Psi}$)}, Bull. Soc. Math. France \ \textbf{134} (2006), 559-631.

\bibitem{lewy}
H.\ Lewy,  \emph{An example of a smooth linear partial differential
  equation without solution}, Ann.\ of Math.\ \textbf{66} (1957), 155--158.

\bibitem{MenUhl}
G.\ A.\ Mendoza and  G.\ Uhlmann, \emph{A necessary condition for local solvability for a class of
              operators with double characteristics}, J.\ Funct.\ Anal.\
            \textbf{52} (1983), 252--256.

\bibitem{mo:solv} 
R.\ D.\ Moyer, \emph{Local solvability in two dimensions: Necessary
conditions for the principal-type case}, Mimeographed manuscript,
University of Kansas, 1978.

\bibitem{nt} 
L.\ Nirenberg and F.\ Treves, \emph{On local solvability
of linear partial differential equations. {P}art {I}: Necessary
conditions}, Comm.  Partial Differential Equations \textbf{23} (1970),
1--38, {\em Part {II}: Sufficient conditions}, Comm. Pure
Appl. Math. {\bf 23} (1970), 459--509; {\em Correction}, Comm. Pure
Appl. Math. {\bf 24} (1971), 279--288.

\bibitem{pop}
P.\ Popivanov and C.\ Georgiev, \emph{Necessary condition for local
  solvability of operators with 
double characteristics}, {Annuaire Univ. Sofia Fac. Math. M\'ec.}
\textbf{75} (1981), 57--71.  

\bibitem{ps:thesis}
K.\ Pravda-Starov, \emph{etude du pseudo-spectre d'op\'erateurs non
  auto-adjoints}, Ph.D. thesis,  Universit\'e de Rennes I, 2006.


\bibitem{trep} 
J.-M.\ Tr\'epreau, \emph{Sur la r\'esolubilit\'e
analytique microlocale des op\'erateurs pseudodiff\'erentiels de type
principal}, Ph.D. thesis, Universit\'e de Reims, 1984.

\end{thebibliography}

\providecommand{\bysame}{\leavevmode\hbox to3em{\hrulefill}\thinspace}

\end{document}